\documentclass[a4paper]{article}

\usepackage[numbers,sort&compress]{natbib}
\usepackage{graphicx,amsmath,amssymb,amsfonts,amsthm,mathtools,enumerate,enumitem}
\usepackage{booktabs,tikz,tabularx,array,xcolor}
\usepackage{subcaption,geometry}
\usepackage{xurl,hyperref}

\newtheorem{theorem}{Theorem}[section]
\newtheorem{lemma}[theorem]{Lemma}

\newtheorem{conjecture}[theorem]{Conjecture}

\theoremstyle{definition}

\hypersetup{
  colorlinks=true,
  linkcolor=red,
  citecolor=blue,
}

\begin{document}


\title{\textbf{Locating a Seymour vertex in oriented graphs having minimum outdegree at most $7$}}

\author{Arpan Sadhukhan\thanks{Department of Mathematics, Indian Institute of Technology Dharwad. \\
Email: {ra.a.sadhukhan@iitdh.ac.in}. ORCID: \href{https://orcid.org/0000-0003-4048-7143}{0000-0003-4048-7143}.}
~~ Sagnik Sen\thanks{Department of Mathematics, Indian Institute of Technology Dharwad. Email: {sen@iitdh.ac.in}.}}


\date{}
\maketitle

\begin{abstract}
A vertex $v$ of an oriented graph $D$ is a 
Seymour vertex 
if the number of outneighbors of $v$ is less than or equal to the number of second outneighbors of $v$ 
(vertices at directed distance exactly $2$ from $v$). 
Seymour's Second Neighborhood Conjecture claims that every oriented graph contains a Seymour vertex. 
We propose the following stronger conjecture (Conjecture 1.4):  
given any oriented graphs $D$, 
there exists a minimum outdegree vertex $s$, 
and a Seymour vertex $v^*$ such that 
the directed distance from $s$ to $v^*$ is at most $\lfloor\frac{\delta^+(D)}{3} \rfloor$, where $\delta^+(D)$ denotes the minimum outdegree of $D$. 

The main contribution of this article is to show that 
Conjecture 1.4 is true for all oriented graphs 
having minimum outdegree at most $7$. 
This, in particular, improves the results of 
Kaneko and Locke (2001) which (implicitly) proved 
Conjecture 1.4 for all oriented graphs having minimum outdegree at most $6$. 
In particular, this yields Seymour's Second Neighborhood Conjecture for this class of oriented graphs. We also construct a family of digraphs attaining the proposed distance bound, showing that it cannot be improved in general.

Finally, by adapting existing arguments, we show that if Seymour's Second Neighborhood Conjecture is true then
given any oriented graph $D$ and any
vertex $v$, it is always possible to find a 
Seymour vertex $v^*$ such that 
the directed distance from $v$ to $v^*$ is at most the outdegree of $v$. 

\medskip

\noindent Keywords: Seymour's Second Neighbourhood Conjecture, minimum out-degree, Seymour vertices' location,  Burckel's conjecture.
\end{abstract}

\section{Introduction}
All graphs (resp., directed graphs) are finite
and simple (resp., oriented) throughout this article. 
Let $D$ be an \textit{oriented graph} which denotes a finite directed graph without loops or digons; that is, there are no vertices $u,v$ for which the arcs $(u,v)$ and $(v,u)$ are both included. The set of vertices and the set of arcs of $D$ are denoted by $V(D)$ and $A(D)$, respectively. 
For a vertex \(u\in V(D)\), let $N_{D}^+(u)=\{v\in V(D): uv\in A(D)\}$ denote the \textit{first out-neighbourhood} or simply the \textit{out-neighborhood} of \(u\), and 
\[ N_{D}^{++}(u)=\{v\in V(D)\setminus (\{u\}\cup N_D^+(u)): \text{ there exists }w\in N_D^+(u)\text{ with }wv\in A(D)\} \] 
denote the \textit{second out-neighbourhood} of \(u\). Moreover, for a vertex $u \in V(D)$, 
let $d_D^+(u) = |N_D^+(u)|$ denote the 
\textit{out-degree} of $u$. 
The minimum out-degree of $D$, denoted by $\delta^+(D)$, is the minimum value of $d_D^+(u)$ where $u$ varies over all the vertices of $D$. A vertex \(u\) is called a \emph{Seymour vertex} if $|N_D^{++}(u)| \ge |N_D^+(u)|$.
We also write $d_D^{++}(u)=|N_D^{++}(u)|$.
Furthermore, 
the \textit{out-distance} from $u$ to $v$, denoted by $d_D^+(u,v)$, is the length of a shortest directed path from $u$ to $v$, if such a path exists.
The subscript of the notation can be dropped if $D$ is clear from the context.

In the early 1990s, Seymour posed the now-celebrated Second Neighborhood Conjecture: 

\begin{conjecture}[Seymour's Second Neighbourhood Conjecture~\cite{DeanLatka1995}]
    Any oriented graph $D$ has a Seymour vertex. 
\end{conjecture}

Seymour's Second Neighbourhood Conjecture was first published by Dean and Latka~\cite{DeanLatka1995} in connection with tournaments. The tournament case, known as Dean's conjecture, was proved by Fisher~\cite{Fisher1996}; Havet and Thomassé~\cite{HavetThomasse2000} later gave a different proof using median orders. The conjecture has since attracted substantial attention; see, for example,~\cite{AiGerkeGutinWangYeoZhou2024,BrantnerBrockmanKaySnively2009, CohnGodboleHarknessZhang2016,DaraFrancisJacobNarayanan2022, EspunyDiazGiraoGranetKronenberg2024,FidlerYuster2007,Ghazal2012, HernandezCruzGaleanaSanchez2012,KanekoLocke2001,LiangXu2017, Llado2013,MezherDaamouch2024,Seacrest2018, guo2026seymourtightorientations, huang2024improvedboundseymourssecond}. The conjecture has remained largely open and is solved for very specific cases. In particular, Kaneko and Locke~\cite{KanekoLocke2001} proved the conjecture for every oriented graph with a vertex of out-degree at most \(6\) in 2001. 
Recently, 
OpenAI has released a preprint presenting a proof of Seymour's
Second Neighborhood Conjecture~\cite{OpenAI2026Seymour},
with an accompanying Lean formalization.

This can be considered as a work in parallel, since an earlier version of this paper was cited in the OpenAI preprint~\cite{OpenAI2026Seymour}. We improve the known minimum outdegree bound from \(6\) to \(7\), proving the conjecture for every oriented graph with minimum out-degree at most \(7\). 
Thus, barring the OpenAI's post, 
to the best of our knowledge, this result is the first improvement in more than two decades to the minimum-out-degree threshold for the Seymour's Second Neighborhood Conjecture. For oriented graphs having minimum out-degree at most $7$, our result proves something slightly stronger than what Seymour's Second Neighborhood Conjecture claims by providing an idea about the location of a Seymour vertex in it.

Observe that, Seymour's Second Neighborhood Conjecture speaks only about the existence of a Seymour vertex, and not about its location. However, all the proof efforts to date (including OpenAI's solution) has attempted to find 
a Seymour vertex ``near" a minimum 
out-degree vertex of an oriented graph, or at least that had been the starting point of the proofs. Therefore, a natural direction to extend the Seymour's Second Neighborhood Conjecture will be to provide some information about the location of a Seymour vertex. However, we observed that it is not possible to always find a Seymour vertex ``near" to a minimum out-degree vertex in the following sense.

\begin{theorem}\label{thm: delta/3 example}
    Let $q$ and $r$ be a positive integers. 
    Then there exists an oriented graph $D_{q,r}$ with minimum out-degree equal to $3r+2$, and exactly $q$ vertices of minimum out-degree satisfying the following:
    given any vertex $v$ with minimum outdegree, and any Seymour vertex $v^*$ in $D_{q,r}$, we have $d^+(v,v^*) > r-1$. 
    \end{theorem}

While the above theorem shows that the 
out-distance of a Seymour vertex 
from any minimum out-degree vertex can be arbitrarily high, 
our next result shows something contrasting and provides information about the location of a Seymour vertex (if it exists) in an oriented graph in terms of its minimum degree.

\begin{theorem}\label{thm: delta bound}
Let $D$ be an oriented graph such that any nonempty induced subdigraph of $D$ contains a Seymour vertex. Then for any vertex $v$, there exists a Seymour vertex $v^*$ in $D$ satisfying 
$d^+(v,v^*) \leq d^+(v)$. 
In particular if $v$ is a vertex of minimum out-degree, there exists a Seymour vertex $v^*$ in $D$ satisfying 
$d^+(v,v^*) \leq \delta^+(D)$. 
\end{theorem}

The argument for the above is not entirely novel, the same argument is already present in Felix’s argument~\cite{Felix2026Seymour} in a slightly different context but to the best of our knowledge our conclusion is particularly not noted anywhere in literature. Also a related shrinking-boundary construction appear in the proof of Proposition 4 of Espuny Díaz et al.~\cite{EspunyDiazGiraoGranetKronenberg2024} Assuming Seymour's conjecture to be correct, the above resolves a conjecture of Burckel recorded as Conjecture 6.14 in~\cite{Sullivan2006}.  

In the proof of the Seymour Conjecture for oriented graphs with $\delta^+(D) \leq 6$ by
Kaneko and Locke~\cite{KanekoLocke2001}, they in fact managed to find the Seymour vertex within out-distace $\lfloor\frac{\delta^+(D)}{3}\rfloor$ of every minimum out-degree vertex $v$ of $D$. This fact is very clear from their proof. We omit the details. This observation, combined with subsequent analysis, motivates us to put forward the following stronger conjecture.

\begin{conjecture}\label{conj: strong}
Let $D$ be an oriented graph. 
Then there exists a minimum out-degree vertex $v$ and a Seymour vertex $v^*$ in $D$ satisfying 
$d^+(v,v^*) \leq \lfloor \frac{\delta^+(D)}{3} \rfloor$. 
\end{conjecture}

Notice that, due to 
Theorem~\ref{thm: delta/3 example}, if the conjecture holds, then the upper bound of 
$\lfloor \frac{\delta^+(D)}{3} \rfloor$ cannot be improved, and is tight in that sense. 
We solve Conjecture~\ref{conj: strong} for all oriented graphs having minimum 
out-degree equal to $7$, which, in particular, implies the Seymour's Second Neighborhood Conjecture for the same class of oriented graphs. This is the main contribution of this paper.

\begin{theorem}\label{thm:main-existence}
Let $D$ be an oriented graph with minimum out-degree at most $7$.
Then there exists a minimum out-degree vertex $v$ and a Seymour vertex $v^*$ in $D$ satisfying 
$d^+(v,v^*) \leq \lfloor \frac{\delta^+(D)}{3} \rfloor$. 
\end{theorem}

We actually believe that an even stronger   conjecture may hold in the following sense.

\begin{conjecture}\label{conj: strong 2}
Let $D$ be an oriented graph. 
Then for all minimum out-degree vertex $v$, there exists a Seymour vertex $v^*$ in $D$ satisfying 
$d^+(v,v^*) \leq \lfloor \frac{\delta^+(D)}{3} \rfloor$. 
\end{conjecture}



\medskip 

\noindent \textbf{Organization:} In Section~\ref{sec:prem}, we introduce some preliminary definitions and notation, and then present the proofs of Theorems~\ref{thm: delta/3 example}, \ref{thm: delta bound}, and~\ref{thm:main-existence} in Sections`\ref{sec: proof delta/3 example}, \ref{sec: proof delta bound}, and~\ref{sec:proof}, respectively. 
Since the proof of Theorem~\ref{thm:main-existence} is long, Section~\ref{sec:proof} is 
divided into multiple appropriate subsections keeping the reader's comfort in mind.

\section{Preliminaries}\label{sec:prem}
Given an oriented graph $D$, an arc from $u$ to $v$ is denoted by $(u,v)$ or $u \to v$, where $u$ is called its  \textit{tail} 
and $v$ is called its  \textit{head}. 
Moreover, if there is an arc from $u$ to $v$, we often use the following phrases to express it for the convenience of writing:
\textit{$u$ sends to $v$}, 
\textit{$u$ dominates $v$}, 
\textit{$v$ receives from $u$},
or \textit{$v$ is dominated by $u$}. 
We also extrapolate the term dominate as follows. 
Let $u, v$ be two vertices, and let $X, Y$ be two sets of vertices. 
The phrase $u$ (resp., $X$) \textit{dominates} $Y$ means $u$ (resp., every vertex of $X$) sends to every vertex of $Y$. 
Similarly,  
the phrase $v$ (resp., $Y$) \textit{is dominated by} $X$ means every vertex of $X$ sends to $v$ (resp., every vertex of $Y$). We say that $X$ \textit{collectively dominates} $Y$ if every vertex of $Y$ receives an arc from at least one vertex of $X$. 

Let $u,v$ be vertices, and let $X,Y$ be  sets of vertices in $D$. Then we define 
$$N^{+}(X)=\bigcup_{x\in X}N^{+}(x), \qquad N_Y^{+}(u) = N^{+}(u) \cap Y, \qquad   N_Y^{+}(X) = N^{+}(X) \cap Y, \qquad N_Y^{++}(u) = N^{++}(u) \cap Y,$$
and
$$ d_Y^{+}(u) = |N_Y^{+}(u)|, \hspace{1cm} 
d_Y^{+}(X)=|N_Y^{+}(X)|, \hspace{1cm}
d_Y^{++}(u)=|N_Y^{++}(u)|, \hspace{1cm} \partial^+X=N^+(X)\setminus X,$$
where $\partial^+X$ is called the 
\textit{out-boundary} of $X$.
In particular, 
if the out-boundary of $X$ is the null-set, that is, if $N^+(X) \subseteq X$, then we say that $X$ is \textit{forward-closed}. 
Moreover, notice that, $N_{D[X]}^{++}(u) \subseteq N_X^{++}(u)$ where $u \in X$ and that equality does not always hold since it is possible for a directed $2$-path to start at $u \in X$ and also end inside $X$ while its internal vertex is outside $X$.  
Finally, $e(X,Y)$ denotes the number of arcs with their tails in $X$ and heads in $Y$. In particular, $e(X,X)$ denotes the number of arcs in $D[X]$.

Now we will state two simple yet important lemmas and a proposition that are already present in the paper~\cite{KanekoLocke2001} as one consolidated lemma. 
 This lemma will be a fundamental tool used to divide the proof of Theorem~\ref{thm:main-existence} into meaningful exhaustive cases.

\begin{lemma}[~\cite{KanekoLocke2001}]\label{lem: old tools}
    Let $D$ be an oriented graph with minimum outdegree at least one and let $s$ be a vertex of $D$ having minimum out-degree. 
    Suppose $N^+(s) = A$, $N^{++}(s) = B$, and $N^+(B) \setminus (A \cup B) = C$. Moreover, let $a_1$ be a vertex of $A$ having minimum 
    $(|N^+(a_1) \cap A|,|N^+(a_1) \cap B|)$ in the lexicographic order. Then the following holds: 
    \begin{enumerate}
        \item[(i)] We have $|N^+(a_1) \cap A| \leq \left\lceil \frac{|A|}{2}\right\rceil-1$.

        \item[(ii)] If $|A|$ is odd and 
        $|N^+(a_1) \cap A|=\frac{|A|-1}{2}$, then $D[A]$ is a out-regular tournament with all vertices having out-degree exactly $\frac{|A|-1}{2}$ in $A$ and $|N_{D[A]}^{++}(a)  \cap A|=|N^{+}(a) \cap A|$ for all $a \in A.$

\item[(iii)] If $|B| \geq |A|$, then $s$ is a Seymour vertex. 

\item[(iv)] If $|B| \leq \left\lceil \frac{|A|}{2} \right\rceil$ and $|A| \geq 2$, then $a_1$ is a Seymour vertex. 
    \end{enumerate}
\end{lemma}

\section{Proof of Theorem~\ref{thm: delta/3 example}}\label{sec: proof delta/3 example}

The construction uses a directed cycle of blocks with somewhat decreasing successive
block sizes. 


\begin{proof}[Proof of Theorem~\ref{thm: delta/3 example}]
For $n\ge1$, let $R_n$ be the oriented graph on $\mathbb Z/n\mathbb Z$
with arcs $x\to x+j$ for $1\le j\le t=\lfloor(n-1)/2\rfloor$. It is easy to observe that first and exact second outdegrees both equal $t$.

Set $\delta=3r+2$. Take disjoint blocks of vertices $V_0,\ldots,V_{r+2}$ of sizes
\begin{equation}\label{sncdist:sizes}
 (n_0,n_1,\ldots,n_{r+2})
 =(q,\delta,\delta-1,\ldots,\delta-r,2\delta+1).
\end{equation}

\emph{Construction:}
Let $V_0$ be an independent set and let $D[V_i]=R_{n_i}$ for $i\ge1$.
Let all the vertices of $V_{i+1}$ are outneighbours of every vertex of $V_i$ for all $i=0,1,2,\ldots, r+1$. Also let all the vertices of $V_{0}$ are outneighbours of every vertex of $V_{r+2}$. No other arcs are present.

Let $t_0=0$ and $t_i=\lfloor(n_i-1)/2\rfloor$ for $i\ge1$.
For $v\in V_i$, it is easy to observe that
\begin{equation}\label{sncdist:counts}
 d_D^+(v)=t_i+n_{i+1},\qquad
 d_D^{++}(v)=t_i+n_{i+2}.
\end{equation}

The block indices in the above equation are taken modulo $r+3$.

Vertices of $V_0$ have outdegree $\delta$. For $1\le i\le r$, if $v_i \in V_i$, then 
\[
 d_D^+(v_i)=\left\lfloor\frac{\delta-i}{2}\right\rfloor+\delta-i
 \ge(r+1)+(2r+2)=\delta+1.
\]
For the final two blocks every vertex in $V_{r+1}$ has outdegree equal to $r+2\delta+1$ and every vertex in $V_{r+2}$ has outdegree equal to $\delta+q$.

Thus precisely the $q$ vertices of $V_0$ have minimum outdegree.
For $0\le i<r$, equation~\eqref{sncdist:counts} gives
$d_D^{++}(v)=d_D^+(v)-1$ on $V_i$ and hence those vertices are not Seymour vertices.
Further every vertex of $V_r$ is Seymour, since $n_{r+2}=2\delta+1>n_{r+1}$.
An arc either stays in its block or advances by one block around the
cycle. Hence any Seymour vertex is at least outdistance $r$ from any vertex of $V_0$. This finishes the proof.
\end{proof}

Thus the conjectured bound in conjectures~\ref{conj: strong},~\ref{conj: strong 2} is attained for $\delta = 3r+2$, hence infinitely often. Thus if the conjectures are true, the bound would be tight.



\section{Proof of Theorem~\ref{thm: delta bound}}\label{sec: proof delta bound}

\begin{proof}[Proof of Theorem~\ref{thm: delta bound}]
For a nonempty subset of vertices $T_i$, put
$\partial^+T_i=(\bigcup_{x\in T_i}N_D^+(x))\setminus T_i$ and
$b(T_i)=|\partial^+T_i|$. Suppose $v_i$ is a Seymour vertex of $D[T_i]$. If $v_i$ is not a Seymour vertex of $D$, set
$P_i=N_D^+(v_i)\setminus T_i$ and $Q_i=N_D^{++}(v_i)\setminus T_i$.
Induced second neighbors remain global second neighbors, so
\[
 d_D^{++}(v_i)\ge d_{D[T_i]}^{++}(v_i)+|Q_i|
 \ge d_{D[T_i]}^+(v_i)+|Q_i|,
 \qquad
 d_D^{++}(v_i)<d_D^+(v_i)=d_{D[T_i]}^+(v_i)+|P_i|.
\]
Thus $|Q_i|<|P_i|$. Now enlarge the set $T_i$ to form $T_{i+1}=T\cup P_i$. Observe that every arc from $P_i$ to outside $T_{i+1}$
ends in $Q_i$. Consequently
\[
 \partial^+T_{i+1}\subseteq(\partial^+T\setminus P_i)\cup Q_i,
 \qquad b(T_{i+1})\le b(T_i)-1.
\]
Start with $T_0=\{v\}$, whose boundary $\partial^+T_0$ has size $d_D^+(v)$.
Each enlargement at every unsuccessful step (when the selected local Seymour vertex is not a global Seymour vertex) adds only outneighbors of a vertex already present.
After $i$ enlargements, every vertex of $T_i$ lies within distance $i$
of $v$. Within $d_D^+(v)$ enlargements, a global Seymour vertex is found or
the boundary is empty. At empty boundary, the Seymour vertex of final induced subgraph $T_j$ (for some $j$) will be a global Seymour vertex. This finishes the proof.
\end{proof}

\section{Proof of Theorem~\ref{thm:main-existence}}\label{sec:proof}
The proof of the Theorem~\ref{thm:main-existence} is lengthy, and thus is presented 
in segments, outlined by subsections, observations, lemmas, etc. 
Recall that, for any oriented graph having minimum out-degree at most $6$, Theorem~\ref{thm:main-existence} is already proved by Kaneko and Locke~\cite{KanekoLocke2001}. Their paper do not explicitly state the result though, however, their proof implicitly imply the same~\cite{KanekoLocke2001}. Thus, we may focus on the instances of oriented graphs having minimum out-degree exactly $7$. 

The proof is by contradiction, and thus, we assume and fix a $D$ that contradicts  Theorem~\ref{thm:main-existence}. That is, assume that $D$ is an oriented graph having minimum outdegree equal to $7$, and no minimum out-degree vertex $s$ of $D$ contains a Seymour vertex inside 
$(\{s\} \cup N^+(s) \cup N^{++}(s))$. 
Fix an arbitrary 
minimum out-degree vertex $s$ to begin with.

\subsection{The framework}
Throughout the proof, unless otherwise stated, let $s$ denote a vertex having out-degree equal to $7$, that is, equal to the minimum out-degree of $D$. Moreover, we assume that
$D$ does not contain a Seymour vertex in 
$(\{s\} \cup N^+(s) \cup N^{++}(s))$. 
This fixed $s$ is referred to as the \textit{root}. 
Moreover, we fix $$A = N^+(s), \hspace{2cm} B = N^{++}(s) = N^+(A) \setminus A, \hspace{2cm} C = N^+(B) \setminus (A \cup B).$$
Note that, $C$ may contain the root $s$ as well. Furthermore, observe that $N^+(B) \subseteq A \cup B \cup C$. 

Furthermore, let us fix $a_1$ as a vertex of $A$ 
having minimum number of out-neighbors in $A$. 
If there are multiple candidates for $a_1$, 
then we choose it in such a way that the number of out-neighbors of $a_1$ in $B$ is also minimized. 
In other words, $a_1$ is a vertex of $A$ where $(d^+_A(a_1),d^+_B(a_1))$ is minimum (with respect to the lexicographic order) among all vertices of $A$. 
To add a few more notation for the proof framework, let us fix 
\[
A_1=N^+(a_1)\cap A,\qquad A_2=N_{D[A]}^{++}(a_1),\qquad
A_3=A\setminus(\{a_1\}\cup A_1\cup A_2).
\]
Here $A_2=N^+(A_1)\cap(A\setminus(\{a_1\}\cup A_1))$. Moreover, let 
$$B_1 = N^+(a_1) \cap B, \hspace{2cm} B_2 = N^{++}(a_1) \cap B, \hspace{2cm} B_3 = B \setminus (B_1 \cup B_2).$$
This is a framework of the proof which we will retain throughout, unless otherwise specified (see Fig.~\ref{fig:proof-framework} for a pictorial reference). 

\begin{figure}
\centering

\resizebox{\textwidth}{!}{%
\begin{tikzpicture}[>=stealth,thick,every node/.style={font=\small}]

\node[draw,circle,inner sep=2.5pt] (s) at (0,0) {$s$};

\draw[rounded corners] (2,-2.4) rectangle (6,2.4);
\node at (4,2.75) {$A=N^+(s)$};

\draw (2,1.25) -- (6,1.25);
\draw (2,0.25) -- (6,0.25);
\draw (2,-0.85) -- (6,-0.85);

\node[draw,circle,inner sep=2pt] (a1) at (3.7,1.75) {$a_1$};

\node at (4,0.75) {$A_1=N^+(a_1)\cap A$};

\node[align=center] at (4,-0.30)
{$A_2=N^{++}_{D[A]}(a_1)$};

\node[align=center,font=\scriptsize] at (4,-1.60)
{$A_3=A\setminus(\{a_1\}\cup A_1\cup A_2)$};

\draw[->] (a1) -- (3.7,1.25);

\draw[->] (3.45,0.55) -- (3.70,0.25);
\draw[->] (4.25,0.55) -- (4.05,0.25);

\draw[rounded corners] (7.8,-2.4) rectangle (12.4,2.4);
\node at (10.1,2.75) {$B=N^{++}(s)$};

\draw (7.8,0.95) -- (12.4,0.95);
\draw (7.8,-0.45) -- (12.4,-0.45);

\node[align=center] at (10.1,1.65)
{$B_1=N^+(a_1)\cap B$};

\node[align=center] at (10.1,0.25)
{$B_2=N^{++}(a_1)\cap B$};

\node[align=center] at (10.1,-1.40)
{$B_3=B\setminus(B_1\cup B_2)$};

\draw[rounded corners] (14.2,-2.4) rectangle (19.2,2.4);
\node[align=center] at (16.7,0)
{$C=N^+(B)\setminus(A\cup B)$};

\node[font=\scriptsize] at (16.7,2.75)
{possibly $s\in C$};


\draw[->, very thick, blue] (s.east) -- (2,0);

\draw[->, very thick, blue] (6,0) -- (7.8,0);

\draw[->, very thick, blue] (12.4,0) -- (14.2,0);

\draw[->] (a1) -- (7.8,1.65);

\end{tikzpicture}
}

\caption{The proof framework around the fixed vertex $s$. We set
$A=N^+(s)$, $B=N^{++}(s)$, and
$C=N^+(B)\setminus(A\cup B)$. The sets $A_1,A_2,A_3$ and
$B_1,B_2,B_3$ denote the corresponding subdivisions determined by
$a_1$. The arrows indicate the relevant reachability relations used
throughout the proof; they do not indicate complete domination between
the corresponding sets.}

\label{fig:proof-framework}
\end{figure}

A few immediate observations follow. 

\begin{lemma}\label{lem: new tools}
    In $D$, we have: 
    \begin{enumerate}
       \item[(i)] $5 \leq |B| \leq 6$;

       \item[(ii)] $1 \leq |A_1| \leq 3$; in particular, $|A_1| \in \{2,3\}$ if $|B| = 5$ and $|A_1| \in \{1,2,3\}$ if $|B| = 6$;
    
        \item[(iii)] $|A_2| \geq |A_1|$;

        \item[(iv)]  $|B_1| \geq 7 - |A_1|$;

        \item[(v)]  if $B=B_1$, 
        then $d^{++}(a_1) \geq |A_2| + |C| \geq |A_1| + |C|$. 
    \end{enumerate}  
\end{lemma}

\begin{proof}
 \noindent (i)  The out-degree of $s$ is $7$, and thus, we have $|A| = 7$. 
   Since $s$ is not a Seymour vertex, $|B| \in \{5, 6\}$ due to Lemma~\ref{lem: old tools}(iii) and~(iv). 

\medskip

\noindent (ii) Lemma~\ref{lem: old tools}(i) implies 
$|A_1| = |N^{+}(a_1) \cap A| \leq 3$. 
As $d^{+}(a_1) \geq 7$, we have 
$$d^+(a_1) = |N^+(a_1) \cap A| + |N^+(a_1) \cap B| \leq |A_1| + |B|, \text{ and thus } |A_1| \geq  d^{+}(a_1) - |B| \geq 7 - |B|.
$$
In particular, $|A_1| \geq 1$ as $|B| \leq 6$. 
Moreover,  if $|B| = 5$, we have $|A_1| \in \{2, 3\}$, and if $|B| = 6$, we have $|A_1| \in \{1, 2, 3\}$. 

\medskip

\noindent (iii) If $|A_1| = 1$, then we know that the vertex in $A_1$, say $a_2$, has at least one out-neighbor in $A$. Since $D$ is oriented, that out-neighbor of $a_2$ in $A$ cannot be $a_1$ or $a_2$. Thus, $|A_2| \geq 1 = |A_1|$ in this case.

If $|A_1| = 2$, then suppose that $A_1 = \{a_2, a_3\}$. We know that each of $a_2, a_3$ has at least two out-neighbors in $A$; since $a_1$ is such a vertex of $A$ whose number of out-neighbors in $A$ is minimum, and in this case is exactly $2$. 
Notice that, the arcs with tails in $A_1$ must have their heads in $A_1 \sqcup  A_2$. Since there can be only one arc having both its tail and head in $A_1$ (as it has exactly two vertices), at least three arcs must have their tails in $A_1$ and head in $A_2$. As $A_1$ has exactly two vertices, a particular vertex in $A_2$ can receive at most two arcs from $A_1$ implying $|A_2| \geq \lceil \frac{3}{2} \rceil = 2 = |A_1|$  in this case. 

If $|A_1| = 3$, then $|A_2| = 3 \geq |A_1|$ by Lemma~\ref{lem: old tools}(ii). 

\medskip

\noindent (iv) Note that, the out-neighbors of $a_1$ belongs to $A \sqcup B$. Also, as per definition, $A_1$ and $B_1$ are  exactly the sets of 
out-neighbors of $a_1$ in $A$ and $B$, respectively. Thus, $N^+(a_1) = A_1 \sqcup B_1$. Since $d^+(a_1) \geq 7$, we have 
$$7 \leq d^+(a_1) = |N^+(a_1)| = |A_1 \sqcup B_1| = |A_1| + |B_1|.$$
This implies $|B_1| \geq 7 - |A_1|$.

\medskip

\noindent (v) Suppose $B = B_1$. That means, $a_1$ dominates $B$. Recall that all out-neighbors of $a_1$ belongs to $A \sqcup B$.
Moreover, by definition, every vertex of $C$ receives at least one arc from a vertex of $B$. thus, every vertex of $C$ is in the second out-neighborhood of $a_1$. 
As $A$ and $C$ are disjoint sets and as the vertices of $A_2$ are second out-neighbors of $a_1$, both  by definition, in particular, we have $d^{++}(a_1) \geq |A_2| + |C|$. The second inequality follows due to part (iii) of the statement (already proved) which shows $|A_2| \geq |A_1|$. 
\end{proof}

The above lemma helps us dividing the proof into four major cases as follows:
\begin{enumerate}
\item[(1)] $|B|=5$, $|A_1| \in \{2,3\}$;
\item[(2)] $|B|=6$, $|A_1|=3$;
\item[(3)] $|B|=6$, $|A_1|=1$;
\item[(4)] $|B|=6$, $|A_1|=2$.
\end{enumerate}
The above four cases are presented in separate subsections in the same order, and also have an increasing order of difficulty and length. In each of the first three cases we arrive to a contradiction to the assumption that $D$ does not contain a Seymour vertex in $(\{s\} \cup A \cup B)$, where $s$ is the initially chosen root. 
In the final case, we refine the choice criteria of the root and update $s$, and then
contradict the existence of $D$ whose minimum outdegree vertices do not contain a Seymour vertex within outdistance $2$ to complete the proof.

\subsection{Frequently used counting arguments}
Let us present three frequently used counting arguments here.
For the first one, let $X_1, X_2, Y_1, Y_2$ be disjoint sets of vertices in $D$. Then 

\medskip

 $\text{\textbf{(R0)}} \hspace{0.5cm} e(X_1 \sqcup X_2, Y_1) = e(X_1, Y_1) + e(X_2, Y_1) \text{ and } e(X_1, Y_1 \sqcup Y_2) = e(X_1, Y_1) + e(X_1, Y_2)$

\medskip

To verify, since the sets are disjoint, any arc from  $(X_1 \sqcup X_2)$ to $Y_1$ must have its tail either in $X_1$ or in $X_2$. Similarly, 
any arc from  $X_1$ to $(Y_1 \sqcup Y_2)$ must have its head either in $Y_1$ or in $Y_2$. Since this is very basic, we will use it without referring to \textbf{(R0)}. 

For the other two formulas, the generic set up is: let $X$ and $Y$ be two disjoint sets of vertices in $D$. We want to count the number of arcs with tail in $X$ and head in $Y$. 
Let $Z = V(D) \setminus (X \sqcup Y)$. With this notation, let us derive two counting rules which we will frequently use in the proof. 
The first rule is:

\medskip

 $\text{\textbf{(R1)}} \hspace{0.5cm} e(X, Y) = \sum_{x \in X} d^+(x) - e(X,X) - e(X, Z) \geq 7|X| - {|X| \choose 2} - e(X,Z).$

\medskip

\noindent To derive this, note that an arc that has its tail in $X$, can have its head in $X$, $Y$, or $Z$. Since, $X, Y, Z$ are pairwise disjoint by definition, the first inequality is justified. For the second inequality, notice that in the summation, each term is $7$ or more since $D$ has minimum out-degree $7$ and the fact that there can be at most 
 ${|X| \choose 2}$ arcs with both their tail and head in $X$. 

 \medskip

 The second rule is:

\medskip

 $\text{\textbf{(R2)}} \hspace{0.5cm} e(X, Y) \leq |X||Y| - e(Y, X).$

\medskip

\noindent To derive this, note that the maximum number of arcs possible between $X$ and $Y$ (any direction) is $|X||Y|$. Thus, $e(X,Y) + e(Y,X) \leq |X||Y|$. The second rule follows by rearranging this inequality.

We will prove a few general lemmas for assisting in our proof using the above two rules.

\begin{lemma}\label{lem: boundary of six is at least 6}
    Let $X$ be a set of vertices in $D$ satisfying $|X| = 6$. If $|\partial^+X| \leq 5$, then $X$ contains a Seymour vertex. 
\end{lemma}

\begin{proof}
    Suppose $|\partial^+X| \leq 5$. Note that, $D[X]$ has at most ${|X| \choose 2} = 15$ arcs, and thus its average out-degree is at most $\frac{15}{6} < 3$. This implies that the minimum out-degree of $D[X]$ is $2$ or less. On the other hand, since any vertex of $X$ can have at most $5$ out-neighbors outside $X$ (since $|\partial^+X| \leq 5)$, it must have at least $2$ out-neighbors inside $X$. That means, the minimum out-degree of $D[X]$ is exactly $2$. 
    
    Let us assume that $x \in X$ is a vertex having exactly $2$ out-neighbors in $X$. Since any vertex of $D$ has out-degree at least $7$, $x$ must dominate $\partial^+X$, and we must have  $|\partial^+X| = 5$. That means, $d^+(x) = 7$.

    Let us now count the number of second out-neighbors of $x$. 
    Since  $d^+_{D[X]}(x) = 2$, and the minimum out-degree of $D[X]$ is exactly $2$, the two 
    out-neighbors of $x$ in $X$ will send at least $4$ arcs inside $X$.  Observe that, at most $1$ of those $4$ arcs can be contained in $N_{D[X]}^{+}(x)$. 
    That means, $N_{D[X]}^{+}(x)$ sends at least $3$ arcs to $N_{D[X]}^{++}(x)$ implying 
    $|N_{D[X]}^{++}(x)| \geq 2$. That is, $X$ contains at least two second out-neighbors of $x$.

    Next, using, \textbf{(R1)} and \textbf{(R2)}, respectively, we have 
    $$e(X, \partial^+X) \geq 7|X| - {|X| \choose 2} = 27 \text{ and } e(\partial^+X, X) \leq |X||\partial^+X| - 27 = 3.$$
    Let $Z = N^+(\partial^+X) \setminus (X \cup \partial^+X)$. Using \textbf{(R1)} again, we get
    $$e(\partial^+X, Z) \geq 7|\partial^+X| - {|\partial^+X| \choose 2} - e(\partial^+X, X) \geq 35 - 10 - 3 = 22.$$
    Therefore, by the pigeonhole principle, 
    $$|Z| \geq \left\lceil \frac{e(\partial^+X, Z)}{|\partial^+X|} \right\rceil \geq \lceil \frac{22}{5} \rceil = 5.$$ 
    That means, $x$ has at least $5$ second out-neighbors in $Z$. Therefore, $x$ is a Seymour vertex.     
\end{proof}

\begin{lemma}\label{lem: counting tools}
    In $D$, the following holds:
    \begin{enumerate}
        \item[(i)] $e(A,B) \geq 28$;

        \item[(ii)] $e(B,A) \leq 7|B| - 28$;

        \item[(iii)] $e(B,C) \geq 18$, if $|B| = 5$ and $e(B,C) \geq 13$, if $|B| = 6$;

        \item[(iv)]  $|C| \geq 4$, if $|B| = 5$ and $|C| \geq 3$, 
        if $|B| = 6$.
    \end{enumerate}
\end{lemma}

\begin{proof}
    \noindent (i) Take $X = A$ and $Y = B$ and apply \textbf{(R1)}. Since any arc from a vertex of $A$ sends to a vertex in $A \sqcup B$, the quantity $e(X,Z)$ of \textbf{(R1)} is equal to $0$. The rest follows from the fact that $|X| = |A| =7$. 

\medskip

\noindent (ii) Take $X = A$ and $Y = B$ and apply \textbf{(R2)}. From part (i) of the proof, we know that, $e(A,B) \geq 28$.  
The rest follows from the fact that $|X| = |A| =7$.

\medskip

\noindent (iii) Observe that any arc from a vertex of $B$ sends to a vertex in $A \sqcup B \sqcup C$. Thus, the quantity $e(X,Z)$ of \textbf{(R1)} is equal to $e(B,A)$ in this case. In the second part of the proof, we showed that $e(B, A) \leq 7|B| - 28$. Therefore, Taking $X = B$ and $Y = C$ and applying \textbf{(R1)} we get 
$$e(B,C) \geq 7|B| - {|B| \choose 2} - e(B, A) \geq 7|B| - {|B| \choose 2} - 7|B| + 28 = 28 - {|B| \choose 2}.$$
The rest follows from replacing $|B|$ by $5$ and $6$ in the above inequality, respectively. 

\medskip

\noindent (iv) Note that, the maximum number of arcs from $B$ to $C$ is trivially upper bounded by $|B||C|$ since $D$ is an oriented graph. In other words, we have 
$$e(B,C) \leq |B||C| \text{ which implies } |C| \geq \frac{e(B, C)}{|B|}.$$
Since $|C|$ is a non-negative integer, the above inequality can be refined as 
$$|C| \geq \left\lceil \frac{e(B, C)}{|B|} \right\rceil.$$ 
The rest follows from
replacing $|B|$ by $5$ and $6$ in the above inequality, respectively.
\end{proof}

\subsection{Case $|B|=5$, $|A_1| \in \{2,3\}$}\label{subsec B=5}
In this subsection, let us assume that 
$|B| = 5$, and then according to Lemma~\ref{lem: new tools}(ii) we know that $|A_1| \in \{2,3\}$. We will show that it is not possible to have $|B|=5$ in $D$ since it forces a Seymour vertex in $A$. 

\begin{lemma}\label{lem: B=5 A_1=2}
In $D$, if $|B|=5$ and $|A_1|=2$, then $a_1$ is a Seymour vertex.
\end{lemma}

\begin{proof}
Suppose, $|B| = 5$ and $|A_1| = 2$. 
Since the minimum out-degree of $D$ is $7$ and $a_1$ has exactly two out-neighbors in $A$, $a_1$ must dominate all the five vertices of $B$. So, $a_1$ must have exactly 
seven out-neighbors, that is, $d^+(a_1) = 7$.  
On the one hand, by Lemma~\ref{lem: new tools}(iii)
$|A_2| \geq |A_1| = 2$ and by Lemma~\ref{lem: counting tools}(iv) $|C| \geq 4$.
On the other hand, 
due to Lemma~\ref{lem: new tools}(v), 
$d^{++}(a_1) \geq |A_2| + |C| \geq 6$. 
 Notice that, if the above inequality is strict, then we are done as $a_1$ will be a Seymour vertex in that scenario. 
Therefore, let us assume $d^{++}(a_1) = |A_2| + |C| = 6$, implying 
$|A_2| =  2$, $|C| = 4$, and thus
$$|A_3| = |A| - |\{a_1\} \sqcup A_1 \sqcup A_2| = |A| - |\{a_1\}| - |A_1| -|A_2| = 7 - 1-2-2 =2.$$ 
Moreover, this also implies that 
$N^{++}(a_1) \cap A = A_2$ (otherwise, $a_1$ will be a Seymour vertex), which in particular means that $B$ does not send any arc to $A_3$.

Now let us recalculate $e(B,A)$ in two ways. Firstly, using the fact that $|C| = 4$, we know $e(B,C) \leq |B||C| \leq 20$. Therefore by \textbf{(R1)} we have, 
\begin{equation}\label{eq: lem B=5 A1=2 counting1}
    e(B,A) \geq 7|B| -  {|B| \choose 2} - e(B,C) \geq 35 - 10 -20 \geq 5.
\end{equation}

Secondly, notice that, since $B$ is dominated by $a_1$, and does not have any arc to $A_3$, 
all arcs from $B$ to $A$ has heads only in $A_1 \sqcup A_2$. 
Also, all arcs from $A_1$ to $A$ has their heads only in $A_1 \sqcup A_2$. 
Hence,

\begin{align}\label{eq: lem B=5 A1=2; counting tool1}
    e(A_1 \sqcup A_2, A) &= e(A_1, A) + e(A_2, A) = e(A_1, A_1 \sqcup A_2) + e(A_2, \{a_1\} \sqcup A_1 \sqcup A_2 \sqcup A_3)  \nonumber \\
    &= e(A_1, A_1) + e(A_1, A_2) + e(A_2, \{a_1\})+ e(A_2, A_1) +e(A_2, A_2) + e(A_2, A_3) \nonumber\\
    &=e(A_1, A_1) + e(A_2, \{a_1\}) +e(A_2, A_2) + e(A_2, A_3) + [e(A_1, A_2)+ e(A_2, A_1)] \nonumber \\
    &\leq 1 + 2 + 1 +4 +4
    = 12   
    \end{align}
since $e(X,Y) + e(Y, X) \leq |X||Y|$ for any disjoint vertex subsets  $X,Y$ by \textbf{(R2)}.

Therefore, by \textbf{(R1)} and the fact that 
any arc from $A_1 \sqcup A_2$ has their heads inside $A \sqcup B$, we have 
$$e(A_1 \sqcup A_2, B) \geq 7|A_1 \sqcup A_2| - e(A_1 \sqcup A_2, A) \geq 28 - 12 = 16.$$
Next using the above relation and \textbf{(R2)} we have,
$$e(B,A) = e(B,A_1 \sqcup A_2) \leq |B||A_1 \sqcup A_2| - e(A_1 \sqcup A_2, B) \leq 20 - 16 = 4,$$
a contradiction to Eq~(\ref{eq: lem B=5 A1=2 counting1}). 
\end{proof}

\begin{lemma}\label{lem: B=5 A_1=3}
In $D$, if $|B|=5$ and $|A_1|=3$, then $A$ contains a Seymour vertex.
\end{lemma}

\begin{proof}
Suppose, $|B| = 5$ and $|A_1| = 3$. 
Thus, by Lemma~\ref{lem: old tools}(ii), $D[A]$ is a tournament where every vertex has out-degree $3$. Due to the minimality criteria of 
$a_1$, it has minimum number of out-neighbors in $B$ among all vertices of $A$. 

That means, if $a_1$ dominates $B$, then every vertex of $A$ dominates $B$ implying $e(B, A) = 0$. 
In such a scenario, using \textbf{(R1)} we have 
$$e(B, C) \geq 7|B| - {|B| \choose 2} - e(B, A) = 35 - 10 = 25 \text{ and thus } 
|C| \geq \frac{e(B,C)}{|B|} \geq 5.$$
Therefore, $d^{+}(a_1) = |A_1|+|B| = 8$, and using the fact $|A_2|\geq |A_1| = 3$ due to Lemma~\ref{lem: new tools}(iii) we have 
$d^{++}(a_1) \geq |A_2|+|C| \geq 8$, making $a_1$ a Seymour vertex.

Hence, we may assume that $a_1$ dominates exactly $4$ vertices $B$, and the vertex that it does not dominate in $B$ is $b_5 \in B$. On the one hand, note that, $d^+(a_1) = |A_1|+|B_1| = 7$, 
where $B_1 = B \setminus \{b_5\}$. 
On the other hand, note that, $|A_2| \geq |A_1| \geq 3$ by Lemma~\ref{lem: new tools}(iii) and
thus, if the vertices of $B_1$ collectively dominate $4$ vertices of $C$, then $a_1$ is a Seymour vertex, and we are done.

Therefore, suppose that $B_1$ collectively dominates at most $3$ vertices of $C$. This implies $e(B_1, C) \leq 12$. As $e(B,C) \geq 18$ by Lemma~\ref{lem: counting tools}(iii), $b_5$ must dominate at least $6$ vertices of $C$. 
By the definition of $B$, there must be some vertex $a \in A$ which dominates $b_5$. Notice that, $d^+(a) \leq |N^+_{D[A]}(a)| + |B| = 8$. Moreover, by Lemma~\ref{lem: old tools}(ii), $a$ has exactly $3$ second out-neighbors in $A$, and at least $6$ second out-neighbors in $C$ (dominated by $b_5$) making $d^{++}(a) \geq 9$. Thus, $a$ is a Seymour vertex.  
\end{proof}

Combining the above lemmas, we can summarize the essential finding of this section 
as follows. 

\begin{lemma}\label{lem: B=5}
    In $D$, it is not possible to have $|B| = 5$. 
\end{lemma}

\begin{proof}
    Follows directly from Lemmas~\ref{lem: B=5 A_1=2} and~\ref{lem: B=5 A_1=3}.
\end{proof}

\subsection{Case $|B|=6$, $|A_1| = 3$}\label{subsec: B=6, A1=3}
In this subsection, let us assume that 
$|B| = 6$ and $|A_1| = 3$. 
Notice that, since $a_1$ has exactly $3$ out-neighbors in $A$, it can have $4, 5$ or $6$ out-neighbors in $B$. 
In other words, $|B_1| \in \{4,5,6\}$. Thus, the values of $|B_2|, |B_3|$ can vary between $0$ and $2$ under the constraint 
$$|B_1| + |B_2| +|B_3| = |B| =6.$$
In particular, the triplet 
\begin{equation}\label{eq: B=6, A1=3: (B1,B2,B3)}
    (|B_1|, |B_2|, |B_3|) \in \{(4,0,2), (4,1,1), (4,2,0), (5,1,0), (5,0,1), (6,0,0) \}.
\end{equation}

We will eventually show that it is not possible to have $|B|=6$ and $|A_1|=3$ in $D$ since it forces a Seymour vertex in $A \sqcup B$ through a series of lemmas.

\begin{lemma}\label{lem: B=6 A_1=3 tool1}
In $D$, let $|B|=6, |A_1|=3$. 
Then we have 
$$e(B_1,C) \geq 7|B_1| - 6|B_2| - 3|B_3| - {|B_1|+|B_2| \choose 2}.$$
Moreover, $B_1$ collectively dominates at least $ \left\lceil \frac{e(B_1, C)}{|B_1|} \right\rceil$ vertices of $C$. 
\end{lemma}

\begin{proof}
    First notice that the vertex $a_1$ does not send any arc to $B_2 \sqcup B_3$, and the $3$ vertices of $A_1$ does not send any arc to $B_3$ by definition. 
    That means, $e(\{a_1\}, B_2 \sqcup B_3) = 0$. Also, recall that $|A_1|=|A_2| = 3$ by Lemma~\ref{lem: old tools}(ii)
    and $e(A_1, B_3) = 0$
    Thus, 
    $$
        e(A, B_2 \sqcup B_3) = 
        e(\{a_1\} \sqcup A_1 \sqcup A_2, B_2 \sqcup B_3) = e(A_1, B_2) +e(A_2, B_2) + e(A_2, B_3) \leq 6|B_2| + 3|B_3|.
    $$
Furthermore, since every vertex of $A$ sends at least $|B_1|$ arcs to $B$, and those arcs  have their head either in $B_1$ or in $B_2 \sqcup B_3$, we have
$$e(A,B_1) \geq 7|B_1| - e(A, B_2 \sqcup B_3).$$
Hence, using \textbf{(R2)} and the above two inequalities we have 
$$e(B_1,A) \leq |A||B_1| - e(A, B_1) \leq 7|B_1| - 7|B_1| + e(A, B_2 \sqcup B_3)  \leq 6|B_2| + 3|B_3|.$$ 
Finally, observe that, the arcs from $B_1$ sends to $A \sqcup B \sqcup C$ only. In fact, inside $B$, it does not send any arc to $B_3$, and thus, sends only to $B_1 \cup B_2$. 
In the worst case scenario, all the edges between $B_1$ and $B_2$ may be oriented towards $B_2$. Therefore, 
$$e(B_1, B) \leq {|B_1|+|B_2| \choose 2}.$$
Thus, compiling all the inequalities that are established till now in this proof, and the fact that every vertex of $B_1$ has out-degree at least $7$, we can show that
$$e(B_1, C) \geq 7|B_1| - e(B_1, A) - e(B_1, B)  \geq  7|B_1| - 6|B_2| - 3|B_3| - {|B_1|+|B_2| \choose 2}. $$
The moreover part follows by observing 
$$|N^{+}(B_1) \cap C| \geq \left\lceil \frac{e(B_1, C)}{|B_1|} \right\rceil$$
by the pigeonhole principle. 
\end{proof}

\begin{lemma}\label{lem: B=6 A1=3 3 cases}
   In $D$, let $|B|=6, |A_1|=3$. If $(|B_1|, |B_2|, |B_3|) \in \{(4,0,2), (4, 1,1), (5, 0, 1)\}$, then $a_1$ is a Seymour vertex. 
\end{lemma}

\begin{proof}
    For $(|B_1|, |B_2|, |B_3|)= (4,0,2), (4, 1,1)$ and $ (5, 0, 1)$, by Lemma~\ref{lem: B=6 A_1=3 tool1} we have $|N^{+}(B_1) \cap C| \geq 4, 3$ and $5$, respectively. 

    Note that, $d^{+}(a_1) = |A_1|+|B_1|$
    and $d^{++}(a_1) \geq |A_2| +|B_2| + |N^{+}(B_1) \cap C|$. 
    That means, in the three cases we are dealing with, $a_1$  has exactly $7$, $7$, or $8$ out-neighbors and at least $7$, $7$, or $8$ second out-neighbors, respectively. That is, $a_1$ is a Seymour vertex in each of these three cases.     
\end{proof}

\begin{lemma}\label{lem: B=6 A1=3 420}
   In $D$, let $|B|=6, |A_1|=3$, and  $(|B_1|, |B_2|, |B_3|) = (4,2,0)$. Then $A$ contains a Seymour vertex. 
\end{lemma}

\begin{proof}
    For $(|B_1|, |B_2|, |B_3|)= (4,2,0)$, by Lemma~\ref{lem: B=6 A_1=3 tool1} we have 
    $|N^{+}(B_1) \cap C| \geq 1$. That means, $d^{+}(a_1) = |A_1|+|B_1| = 7$
    and $d^{++}(a_1) \geq |A_2| +|B_2| + |N^{+}(B_1) \cap C| = 5+|N^{+}(B_1) \cap C|$. 
    That means, $a_1$ is a Seymour vertex unless 
    $|N^{+}(B_1) \cap C| = 1$. 
    Thus, from now on in the proof, let us assume that $|N^{+}(B_1) \cap C| = 1$.

    Notice that, $B_1$ can send at most $|B_1||N^{+}(B_1) \cap C| = 4$ arc to $C$. That means, $B \setminus B_1$ sends at least $e(B,C) -4$ arcs to $C$. 
    As $B \setminus B_1$ has exactly $2$ vertices, at least one of them, say $b$, sends at least $\left\lceil \frac{e(B,C)-4}{2} \right\rceil$ arcs to $C$. 
    
    Let $a \in A$ be vertex that dominates $b$ (such a vertex exists due to the definition of $B$). 
    Suppose that $a$ has exactly $(4+m)$ out-neighbors in $B$. Since every vertex of $A$ has at least $4$ out-neighbors in $B$, and as $|B| =6$, the values of $m$ is between $\{0,1,2\}$. 

    As all vertices of $A$ sends at least $4$ arcs to $B$ while the particular vertex $a$ send $m$ extra, we have $e(A,B) \geq 4|A|+m = 28+m$. That implies, 
    $$e(B, A) \leq |A||B| - e(A,B) \leq 42 - 28-m = 14-m$$ 
    by \textbf{(R2)}. Therefore by \textbf{(R1)}, 
    $$e(B,C) \geq 7|B|-{|B| \choose 2} -e(B,A) \geq 13+m.$$ 
    Thus, $b$ dominates at least $\left\lceil \frac{13+m-4}{2} \right\rceil = 4 + \lceil \frac{m+1}{2} \rceil$ vertices of $C$. 
    That means, $a$ has exactly $d^+(a) = 7+m$ 
    out-neighbors ($3$ in $A$ and $4+m$ in $B$) and at least
    $d^{++}(a) \geq 7+\lceil \frac{m+1}{2} \rceil$ second out-neighbors ($3$ in $A$, and the rest in $C$). 
    Hence $a$ is a Seymour vertex.  
\end{proof}

\begin{lemma}\label{lem: B=6 A1=3 510}
   In $D$, let $|B|=6, |A_1|=3$, and  $(|B_1|, |B_2|, |B_3|) = (5,1,0)$. Then $A$ contains a Seymour vertex. 
\end{lemma}

\begin{proof}
    For $(|B_1|, |B_2|, |B_3|)= (5,1,0)$, by Lemma~\ref{lem: B=6 A_1=3 tool1} we have 
    $|N^{+}(B_1) \cap C| \geq 3$. That means, $d^{+}(a_1) = |A_1|+|B_1| = 8$
    and $d^{++}(a_1) \geq |A_2| +|B_2| + |N^{+}(B_1) \cap C| = 4+|N^{+}(B_1) \cap C|$. 
    That means, $a_1$ is a Seymour vertex unless 
    $|N^{+}(B_1) \cap C| = 3$. 
    Thus, from now on in the proof, let us assume that $|N^{+}(B_1) \cap C| = 3$. In particular, this implies $e(B_1, C) \leq 15$.

 Suppose that $B \setminus B_1 = \{b\}$, and that $b$ 
 is dominated by $m_1$ vertices of $A$ and $m_2$ vertices of $B$. We know that at least one such vertex, say $a$, exists by the definition of $B$. Further assume that $a$ sends exactly $5+m$ arcs to $B$, where $m \in \{0,1\}$. 
 
 Therefore, since each vertex of $A$ sends at least $5$ arcs to $B$, we have $e(A,B_1) \geq 5|A|+m - m_1 = 35+m-m_1$. By \textbf{(R2)} we have, 
 $e(B_1, A) \leq |A||B_1| - e(A, B_1) \leq m_1-m$. 
 Furthermore, we also have $e(B_1, B) = e(B_1, B_1) + e(B_1, \{b\}) \leq 10+m_2$. Thus, overall, $e(B_1, A \sqcup B) \leq 10+m_1+m_2-m$. 

Since $B_1$ sends arcs only to $A \sqcup B \sqcup C$, and as each vertex of $B_1$ has out-degree at least $7$, we have 
$e(B_1, C) \geq 7|B_1| - e(B_1, A \sqcup B) \geq 35 - 10+m - m_1 - m_2 = 25+m - m_1 - m_2$. However, we have already established that $e(B_1, C) \leq 15$. This implies $m_1+m_2 \geq 10+m$. 
That means, out of the $12$ vertices in 
$A \sqcup B_1$, $b$ is dominated by at least $10+m$, and thus can send at most $2-m$ arcs to $A \sqcup B_1$. That also means, $b$ dominates at least $5+m$ vertices of $C$.

Observe that, $d^+(a) = 3+5+m=8+m$ and $d^{++}(a) \geq 3+5+m = 8+m$. Hence $a$ is a Seymour vertex. 
\end{proof}

\begin{lemma}\label{lem: B=6 A1=3 600}
   In $D$, let $|B|=6, |A_1|=3$, and  $(|B_1|, |B_2|, |B_3|) = (6,0,0)$. Then $A \cup B$ contains a Seymour vertex. 
\end{lemma}    
    
  \begin{proof}
    For $(|B_1|, |B_2|, |B_3|)= (6,0,0)$, by Lemma~\ref{lem: B=6 A_1=3 tool1} we have 
    $|N^{+}(B_1) \cap C|  \geq 5$. 
    That means, $d^{+}(a_1) = |A_1|+|B_1| = 9$
    and $d^{++}(a_1) \geq |A_2| +|B_2| + |N^{+}(B_1) \cap C| = 3+|N^{+}(B_1) \cap C|$. 
    That means, $a_1$ is a Seymour vertex unless 
    $|N^{+}(B_1) \cap C| = 5$. 
    Thus, from now on in the proof, let us assume that $|N^{+}(B_1) \cap C| = 5$.

    Since $a_1$ dominates $B$, $A$ dominates $B$ and $B = B_1$. 
    Thus, $\partial^+B = \partial^+B_1 = N^{+}(B_1) \cap C$. As $|B| = 6$ and
    $|\partial^+B| = |N^{+}(B_1) \cap C| = 5$, applying Lemma~\ref{lem: boundary of six is at least 6} we know that there must be a Seymour vertex in $B$. 
\end{proof}

Combining the above lemmas, we can summarize the essential finding of this section 
as follows.

\begin{lemma}\label{lem: B=6 A1=3}
 In $D$, it is not possible to have $|B| = 6$ and $|A_1|=3$.     
\end{lemma}

\begin{proof}
   If $|B|=6$ and $|A_1|=3$, then by Eq~(\ref{eq: B=6, A1=3: (B1,B2,B3)}) we know that 
     $$(|B_1|, |B_2|, |B_3|) \in \{(4,0,2), (4,1,1), (4,2,0), (5,1,0), (5,0,1), (6,0,0) \}.$$
     Lemma~\ref{lem: B=6 A1=3 3 cases} shows that $(|B_1|, |B_2|, |B_3|) \not\in \{(4,0,2), (4,1,1),  (5,0,1) \}$. Moreover, Lemmas~\ref{lem: B=6 A1=3 420}, \ref{lem: B=6 A1=3 510}, and~\ref{lem: B=6 A1=3 600} shows that 
     $(|B_1|, |B_2|, |B_3|) \not\in \{(4,2,0), (5,1,0),  (6,0,0) \}$. Thus, it is not possible to have $|B|=6$ and $|A_1|=3$ in $D$. 
\end{proof}

\subsection{Case $|B|=6$, $|A_1| = 1$}\label{subsec: B=6, A1=1}
In this subsection, let us assume that 
$|B| = 6$ and $|A_1| = 1$. 
Notice that, since $a_1$ has exactly $1$ out-neighbors in $A$, it must dominate all the $6$ vertices of $B$ and have $d^+(a_1) = 7$. Thus, all the vertices of $C$ are second out-neighbors of $a_1$. In particular, $N^{++}(a_1) = N^{++}_A(a_1) \cup C$. 
Thus, if $|N^{++}_A(a_1)|+|C| \geq 7$, $a_1$ must be a Seymour vertex contradicting the assumption that $a_1$ is not a Seymour vertex. Thus, we may maintain 
$|N^{++}_A(a_1)|+|C| \leq 6$.

Recall that, $A_2 = N^{++}_{D[A]}(a_1) \subseteq N^{++}_A(a_1)$. Also, by Lemma~\ref{lem: new tools}(iii), we have $|A_2| \geq |A_1| = 1$ and by Lemma~\ref{lem: counting tools}(iv) we have $|C| \geq 3$. Thus, as of now, the possible values of the ordered pair is
\begin{equation}\label{eq: B=6, A1=1: (N++, C)}
    (|N^{++}_A(a_1)|, |C|) \in \{(1,5), (1,4), (1,3), (2,4), (2,3), (3,3)\}.
\end{equation}
First we are going to show that only two of these values are feasible.

\begin{lemma}\label{lem: B=6 A1=1 4 cases}
 In $D$, let $|B|=6, |A_1|=1$. It is not possible to have 
$(|N^{++}_A(a_1)|, |C|) \in \{(1,5), (1,4), (1,3),  (2,3)\}.$
\end{lemma}

\begin{proof}
    Suppose $A_1 = \{a_2\}$. 
    
    First let us assume that $|N^{++}_A(a_1)|=1$, and 
    $N^{++}_A(a_1) = \{a_3\}$. 
    This, in particular, means that $a_2$ has exactly $1$ out-neighbor in $A$,  which is $a_3$. This implies that $a_1$ and $a_2$ both dominates all the $6$ vertices of $B$, and has all the vertices of $C$ as its second out-neighbor.  
    Moreover, any arc from $B$ to $A$ must have either $a_2$ or $a_3$ as its head. However, since $a_2$ dominates $B$, the only out-neighbor of the vertices of $B$ in $A$ can be $a_3$. 
    That means, $\partial^+B \subseteq C \cup \{a_3\}$
    which implies $|\partial^+B| \leq |C| +1$. 
    Hence, it is not possible to have $|C| \leq 4$ due to Lemma~\ref{lem: boundary of six is at least 6} implying $|C| = 5$.

    If $|C| = 5$ and $a_3$ has at least $2$ out-neighbors in $A$, then $a_2$ is a Seymour vertex. That implies, $a_3$ has exactly $1$ out-neighbor in $A$, and thus, dominates all vertices of $B$. In this scenario, $B$ does not send any arc to $A$. That means, $\partial^+B = C$, which is not possible due to Lemma~\ref{lem: boundary of six is at least 6}. 

The above shows that it is not possible to have $|N^{++}_A(a_1)|=1$. Therefore, for the rest of the proof, let us assume that $|N^{++}_A(a_1)|=2$ and thus, suppose $(|N^{++}_A(a_1)|, |C|) = (2,3)$. 
Notice that, $a_2$ does not send any arc to $A \setminus N^{++}_A(a_1)$ and itself. Therefore, the number of arcs from $\{a_2\} \sqcup N^{++}_A(a_1) = A'$ (say)
to $A$ can be bounded as 
$$e(A', A) \leq {3 \choose 2} + (3-1)(7-3) = 11.$$
Next, using \textbf{(R1)} we have 
$$e(A', B) \geq 7|A'| - e(A', A) \geq 21 - 11 = 10$$
and 
\begin{align*}
e(B,C)&\geq7|B|-\binom{|B|}{2}-e(B,A)
\geq42-15-e(B,A')\\
&\geq27-(|A'||B|-e(A',B))\geq27-(18-10)=19.
\end{align*}
Hence, $$|C| \geq \left\lceil \frac{e(B,C)}{|B|}\right\rceil \geq 4,$$
a contradiction. 
\end{proof}

That leaves us with two cases: 
$(|N^{++}_A(a_1)|, |C|) = (2,4)$ and $(3,3)$.

Before dealing with the remaining two cases, let us record a counting argument which will also be useful later.

\begin{lemma}\label{snc:lem:common-block}
Let $X$ and $Y$ be disjoint nonempty sets of vertices in $D$. Suppose that $D[X]$ is a tournament, every vertex of $X$ has out-degree exactly $7$ in $D$, and $X$ dominates $Y$. Then some vertex $v\in X$ satisfies
\[
 d^{++}(v)\geq |X|+\left\lfloor\frac{|Y|}{2}\right\rfloor.
\]
In particular, $X$ contains a Seymour vertex if the right-hand side is at least $7$.
\end{lemma}
\begin{proof}
Let $v$ be a vertex having maximum out-degree in $D[X]$. Every vertex of $X\setminus(\{v\}\cup N^+(v))$ is a second out-neighbor of $v$ inside $D[X]$. Indeed, otherwise such a vertex would dominate $v$ and all its out-neighbors in $X$, contradicting the choice of $v$.

Let $Z=N^+(Y)\setminus(X\cup Y)$. Since $X$ dominates $Y$, there is no arc from $Y$ to $X$. Thus, by \textbf{(R1)},
\[
 e(Y,Z)\geq 7|Y|-\binom{|Y|}{2},\qquad
 |Z|\geq 7-\left\lfloor\frac{|Y|-1}{2}\right\rfloor.
\]
The vertex $v$ has exactly $7-|N^+(v)\cap X|-|Y|$ out-neighbors outside $X\cup Y$. All the other vertices of $Z$ are second out-neighbors of $v$ through $Y$. These are disjoint from its second out-neighbors in $X$. Hence,
\begin{align*}
 d^{++}(v)&\geq |X|-1-|N^+(v)\cap X|
 +7-\left\lfloor\frac{|Y|-1}{2}\right\rfloor
 -(7-|N^+(v)\cap X|-|Y|)\\
 &=|X|+\left\lfloor\frac{|Y|}{2}\right\rfloor.
\end{align*}
\end{proof}

\begin{lemma}\label{snc:lem:ten-returns}
In $D$, if $|B|=6$ and $|C|=3$, then $e(B,A)\geq10$.
\end{lemma}
\begin{proof}
By \textbf{(R1)}, we have $e(B,A)\geq42-15-18=9$. If equality holds, then every vertex of $B$ has out-degree $7$, $D[B]$ is a tournament, and $B$ dominates $C$. Lemma~\ref{snc:lem:common-block}, with $X=B$ and $Y=C$, gives a Seymour vertex in $B$, a contradiction. Therefore, $e(B,A)\geq10$.
\end{proof}

\begin{lemma}\label{lem: B=6, A1=1, A2=N++(a_1)=N+(B), N+(a2)=B}
In $D$, suppose $|B|=6$, $A_1=\{a_2\}$, and
$(|N_A^{++}(a_1)|,|C|)=(2,4)$. Then
\[
 A_2=N_A^{++}(a_1)=N^+(B)\cap A,
 \qquad N^+(a_2)\cap B=B.
\]
\end{lemma}
\begin{proof}
Let $N_A^{++}(a_1)=\{a_3,a_4\}$. Recall that $a_1$ dominates $B$, every arc from $B$ to $A$ has its head in $\{a_2,a_3,a_4\}$, and every out-neighbor of $a_2$ in $A$ belongs to $\{a_3,a_4\}$.

First suppose that $a_2$ has exactly one out-neighbor in $A$, say $a_3$. Then $a_2$ dominates $B$ and has out-degree $7$. Consequently,
\[
 \partial^+B\subseteq C\sqcup\{a_3,a_4\}.
\]
Lemma~\ref{lem: boundary of six is at least 6} implies that both $a_3$ and $a_4$ receive an arc from $B$. Thus $a_3$ has at least two out-neighbors in $A$. The second out-neighborhood of $a_2$ contains $C$, the vertex $a_4$, and every out-neighbor of $a_3$ in $A$. Since $a_2$ is not a Seymour vertex, we must have
\[
 N^+(a_3)\cap A=\{a_4,a_5\}
\]
for some $a_5\in A\setminus\{a_2,a_3,a_4\}$. Hence $a_3$ has out-degree $7$ and exactly five out-neighbors in $B$. Let $B'=N^+(a_3)\cap B$ and $B\setminus B'=\{b\}$.

By \textbf{(R1)}, at least $42-15-24=3$ arcs go from $B$ to $\{a_3,a_4\}$. At most one of them enters $a_3$, so at least two enter $a_4$. It follows that $a_4$ has at least three out-neighbors in $A$. Since $a_3$ sends to $a_4$, $a_4$ has at most $5$ out-neighbors in $A$, and thus at least $2$ in $B$. Therefore,
\[
 e(B,A)\leq1+4=5,\qquad e(B,C)\geq42-15-5=22.
\]
At most two of the $24$ possible arcs from $B$ to $C$ are absent. In particular, $B'$ collectively dominates all four vertices of $C$.

At least two of the out-neighbors of $a_4$ in $A$ are different from $a_5$. They are second out-neighbors of $a_3$. Together with $C$, these give six second out-neighbors of $a_3$. Thus $b$ cannot be a second out-neighbor of $a_3$, and no arc from $B'\cup\{a_4,a_5\}$ ends in $b$. Also, any arc from $B'$ to $A$ must end in $a_4$. By \textbf{(R1)},
\[
 e(B',\{a_4\})\geq35-10-20=5.
\]
That means, $B'$ dominates $a_4$. The vertex $a_4$ sends to neither $B'$ nor $b$, contradicting the fact that it has at least two out-neighbors in $B$.

Therefore, $N^+(a_2)\cap A=\{a_3,a_4\}$. It remains to show that $a_2$ dominates $B$. Suppose instead that $B'=N^+(a_2)\cap B$ has five vertices, and let $B\setminus B'=\{b\}$. For this part of the proof, put $A'=\{a_2,a_3,a_4\}$. The vertex $a_2$ sends no arc to $A\setminus A'$, so
\[
 e(A',A)\leq3+2\cdot4=11,\qquad
 e(A',B)\geq21-11=10.
\]
Hence $e(B,A)=e(B,A')\leq18-10=8$, and $e(B,C)\geq42-15-8=19$.

Suppose that $B'$ does not collectively dominate some vertex of $C$. Then $e(B,C)\leq1+3\cdot6=19$, and equality holds throughout
\[
 42\leq e(B,B)+e(B,A')+e(B,C)\leq15+8+19.
\]
In particular, $e(A',B)=10$ and $e(A',A)=11$. Thus both $a_3$ and $a_4$ dominate all four vertices of $A\setminus A'$, while $B'$ collectively dominates the other three vertices of $C$. These seven vertices are second out-neighbors of $a_2$, a contradiction.

Therefore, $B'$ collectively dominates $C$. At least three arcs from $B$ enter $A'$, and at most one enters $a_2$. For each $i\in\{3,4\}$, we have
\[
 e(B,\{a_i\})\leq6-|N^+(a_i)\cap B|
 \leq |N^+(a_i)\cap A|-1.
\]
Consequently, $e(\{a_3,a_4\},A)\geq4$. At most one of these arcs stays in $\{a_3,a_4\}$, and none ends in $a_2$. Thus $a_2$ has at least two second out-neighbors in $A$. Together with $C$, these give six. It follows that no arc from $B'\cup\{a_3,a_4\}$ ends in $b$.

Now \textbf{(R1)} gives
\[
 e(B',\{a_3,a_4\})\geq35-10-20=5.
\]
On the other hand, $a_2$ has at most two second out-neighbors in $A$, so
\[
 e(\{a_3,a_4\},A)\leq1+2\cdot2=5.
\]
Thus at least $14-5=9$ arcs go from $\{a_3,a_4\}$ to $B$, all with their heads in $B'$. By \textbf{(R2)},
\[
 e(B',\{a_3,a_4\})\leq10-9=1,
\]
a contradiction. Hence $a_2$ dominates $B$. The only possible heads in $A$ of arcs from $B$ are now $a_3,a_4$, and both receive such an arc by Lemma~\ref{lem: boundary of six is at least 6}. This proves the statement.
\end{proof}

\begin{lemma}\label{lem: one four impossible}
In $D$, it is not possible to have $|B|=6$, $|A_1|=1$, and $(|N_A^{++}(a_1)|,|C|)=(2,4)$.
\end{lemma}
\begin{proof}
Suppose otherwise. By the preceding lemma, we may write
\[
 A_1=\{a_2\},\qquad A_2=N_A^{++}(a_1)=N^+(B)\cap A = \{a_3,a_4\},
 \]
and both $a_1$ and $a_2$ dominate $B$. Recall that, for $i\in\{3,4\}$, let $d_{A}^+(a_i)$ and $d_{B}^+(a_i)$ denote the numbers of out-neighbors of $a_i$ in $A$ and $B$, respectively. We use these two counts only in this proof. Since $a_2\to a_i$ and $a_i$ receives an arc from $B$,
\[
 2\leq d_{A}^+(a_i)\leq5,\qquad2\leq d_{B}^+(a_i)\leq5,\qquad d_{A}^+(a_i)+d_{B}^+(a_i)\geq7.
\]
Also, by \textbf{(R1)} and \textbf{(R2)},
\begin{equation}\label{eq: one four missing}
24 \geq  e(B,C)\geq42-15-(12-d_{B}^+(a_3)-d_{B}^+(a_4))=15+d_{B}^+(a_3)+d_{B}^+(a_4).
\end{equation}
Thus at most $9-d_{B}^+(a_3)-d_{B}^+(a_4)$ of the possible arcs from $B$ to $C$ are absent. Notice that a nonempty subset $B'$ of $B$ can fail to collectively dominate a vertex of $C$ only if all $|B'|$ corresponding arcs are absent.

We first establish four observations.

\medskip\noindent (i) We have $d_{B}^+(a_3)+d_{B}^+(a_4)\leq8$. Indeed, Eq.~\eqref{eq: one four missing} first gives $d_{B}^+(a_3)+d_{B}^+(a_4)\leq9$. Equality forces $D[B]$ to be a tournament, all vertices of $B$ to have out-degree $7$, and $B$ to dominate $C$. Lemma~\ref{snc:lem:common-block} gives a Seymour vertex, a contradiction.

\medskip\noindent (ii) Suppose $d_{B}^+(a_3)=d_{B}^+(a_4)=4$. Then $N^+(a_i)\cap B$ collectively dominates both vertices of $B\setminus N^+(a_i)$, for each $i\in\{3,4\}$. To prove this for $a_3$, put $B'=N^+(a_3)\cap B$ and $B\setminus B'=\{b,b'\}$. Suppose no vertex of $B'$ sends to $b$. Its outgoing arcs then satisfy
\[
 28\leq e(B',B')+e(B',\{b'\})+e(B',C)+e(B',\{a_4\})
 \leq6+4+16+2=28.
\]
Equality holds throughout. Thus $B'$ dominates $b'$, and exactly two vertices of $B'$ send to $a_4$. The vertex $b'$ has at most one out-neighbor in $B$. Its minimum out-degree forces it to dominate $a_3,a_4$ and $C$. At least three arcs from $B$ now enter $a_4$, contradicting $d_{B}^+(a_4)=4$. The proof for $a_4$ is similar. 

\medskip\noindent (iii) If $(d_{A}^+(a_3),d_{B}^+(a_3))=(4,3)$ and $d_{B}^+(a_4)\geq4$, then $a_3$ is a Seymour vertex. Indeed, Eq.~\eqref{eq: one four missing} shows that $B'=N^+(a_3)\cap B$ collectively dominates $C$. It also collectively dominates at least two vertices of $B\setminus B'$. Otherwise,
\[
 21\leq e(B',B')+e(B',B\setminus B')+e(B',C)+e(B',\{a_4\})
 \leq3+3+12+2=20,
\]
a contradiction. If an out-neighbor of $a_3$ in $A$ dominates $B$, then $a_3$ has all three vertices of $B\setminus B'$ and all four vertices of $C$ as second out-neighbors. Otherwise, each of its four out-neighbors in $A$ has at least two out-neighbors in $A$. Their at least eight outgoing arcs cannot all stay in a vertex set of size $4$. None ends in $a_3$, so $a_3$ has a second out-neighbor in $A$, in addition to two in $B$ and four in $C$. In either case, it is a Seymour vertex. The same argument applies with $a_3,a_4$ interchanged.

\medskip\noindent (iv) If $(d_{A}^+(a_3),d_{B}^+(a_3))=(5,2)$, then $a_3$ is a Seymour vertex. Its out-neighbors in $A$ are exactly $A\setminus\{a_2,a_3\}$ and include $a_1$. Through $a_1$, it reaches $a_2$ and its four missing vertices of $B$. By Eq.~\eqref{eq: one four missing}, at most five arcs from $B$ to $C$ are absent. Its two out-neighbors in $B$ therefore collectively dominate at least two vertices of $C$, giving seven second out-neighbors. Again, the argument is symmetric.

\medskip
Next, without loss of generality, let us assume that $d_{B}^+(a_3) = 3$ and $d_{B}^+(a_4) \geq 4$. Using (iii) we know that 
$d_{A}^+(a_3) = 4$ is not possible. 
Hence, suppose, $d_{A}^+(a_3) = 5$. Thus, we have 
$d^+(a_3) = 8$. On the other hand, since $a_3$ dominates all vertices of $A$ except $a_2$ (as $a_2$ sends to $a_3$) and itself, $a_3$ must send to $a_1$, and thus must have $a_2$ as a second out-neighbor. Moreover, since $a_1$ dominates $B$, $a_3$ has $3$ 
second out-neighbors in  $B$. 
By Eq.~\eqref{eq: one four missing}, we know that there are at most $2$ missing arcs from $B$ to $C$, and thus, apart from the 
above-mentioned, $a_3$ has $4$ second out-neighbors in $C$. This forces $a_3$ to be a Seymour vertex. Thus, it is not possible to have $d_{B}^+(a_3) = 3$ and $d_{B}^+(a_4) \geq 4$. By symmetry, we can say it is not possible to have $d_{B}^+(a_3) \geq 4$ and $d_{B}^+(a_4) = 3$.

Further assume that $d_{B}^+(a_3)=d_{B}^+(a_4)=3$. Without loss of generality assume that $a_3$ does not send to $a_4$. Then we must have
\[
 d_{A}^+(a_3)=4,\qquad N^+(a_3)\cap A=A\setminus\{a_2,a_3,a_4\}.
\]
In particular, $a_3$ dominates $a_1$. Thus, $a_3$ has $a_2$ and three vertices of $B$ as its second out-neighbors. Its three out-neighbors in $B$ collectively dominate at least three vertices of $C$, again making it a Seymour vertex.

That means, it is not possible to have 
$d_{B}^+(a_3) \leq 3$ or 
$d_{B}^+(a_4) \leq 3$. 
Therefore, by (i), we must have 
$d_{B}^+(a_3)=d_{B}^+(a_4)=4$.

Suppose $d_{A}^+(a_3)=3$. By (ii) and Eq.~\eqref{eq: one four missing}, $a_3$ has two second out-neighbors in $B$ and four in $C$. Thus, it cannot have another in $A$. 
So $N_{D[A]}^+(a_3)$ forms a forward-closed set in $A$. 
Every vertex of $A$ has at least one 
out-neighbor in $A$, so $N_{D[A]}^+(a_3)$ induces a directed triangle, and each of its vertices dominates $B$ and has out-degree $7$. The triangle contains neither $a_3$ nor $a_4$. Any vertex of the triangle has $a_3,a_4$, all four vertices of $C$, and one vertex of the triangle as second out-neighbors, a contradiction.
Hence $d_{A}^+(a_3) \geq 4$, and by symmetry, $d_{A}^+(a_4)\geq4$.

Without loss of generality assume that 
$a_3$ does not send to $a_4$. Then $d_{A}^+(a_3)=4$ and $d^+(a_3)=8$. 
Let $B'=N^+(a_3)\cap B$. The vertex $a_3$ has the two vertices of $B\setminus B'$, all four vertices of $C$, and $a_2$ as second out-neighbors. Therefore, no vertex of $B'$ sends to $a_4$. It also cannot send to $a_3$, so $e(B',A)=0$. Consequently,
\[
 e(B',B\setminus B')\geq 7|B'| - {|B'| \choose 2} - e(B',C) \geq 28 - 6 - 16 = 6,\qquad
 e(B\setminus B',B)\leq1+8-6=3.
\]
Some vertex $b\in B\setminus B'$ has at most one out-neighbor in $B$. Its minimum out-degree forces it to have exactly one, to dominate $a_3,a_4$ and $C$, and to have out-degree $7$. Through $a_3$, it reaches all four vertices of $A\setminus\{a_2,a_3,a_4\}$ and at least three vertices of $B'$. These are seven distinct second out-neighbors of $b$, a contradiction.
\end{proof}

So, we are left with the case $(|N_A^{++}(a_1)|,|C|)=(3,3)$.

\begin{lemma}\label{lem: one three impossible}
In $D$, it is not possible to have $|B|=6$, $|A_1|=1$, and $(|N_A^{++}(a_1)|,|C|)=(3,3)$.
\end{lemma}
\begin{proof}
Suppose otherwise, and write $A_1=\{a_2\}$. Let
$A'=\{a_2\}\sqcup N_A^{++}(a_1)$. Every arc from $B$ to $A$ has its head in $A'$, while $a_2$ sends no arc to $A\setminus A'$. Thus,
\[
 e(A',A)\leq6+3\cdot3=15,\qquad
 e(A',B)\geq28-15=13.
\]
It follows that $e(B,A)\leq24-13=11$ and $e(B,C)\geq42-15-11=16$. In particular, any three vertices of $B$ collectively dominate $C$. Also, Lemma~\ref{snc:lem:ten-returns} gives $e(B,A)\geq10$.

Observe that, $1\leq d_{A}^+(a_2)\leq3$, $d_{B}^+(a_2)\leq6$, and $d_{A}^+(a_2)+d_{B}^+(a_2)\geq7$. Counting arcs from $N_A^{++}(a_1)$ to itself, to $A\setminus A'$, and to $a_2$, respectively, gives
\begin{equation}\label{eq: one three A count}
 e(N_A^{++}(a_1),A)\leq3+9+(3-d_{A}^+(a_2))=15-d_{A}^+(a_2).
\end{equation}
Moreover,
\[
 10\leq e(B,A)\leq(6-d_{B}^+(a_2))+18-e(N_A^{++}(a_1),B).
\]
Therefore,
\begin{equation}\label{eq: one three B count}
 e(N_A^{++}(a_1),B)\leq14-d_{B}^+(a_2),\qquad
 e(N_A^{++}(a_1),A)\geq21-(14-d_{B}^+(a_2))=7+d_{B}^+(a_2).
\end{equation}
Combining these inequalities, we get $7\leq d_{A}^+(a_2)+d_{B}^+(a_2)\leq8$. So, there are four cases to consider.

\medskip\noindent (i) Suppose $d^+(a_2)=8$. Then $d_{A}^+(a_2)\geq2$, and equality holds in Eqs.~\eqref{eq: one three A count} and~\eqref{eq: one three B count}. Thus $D[N_A^{++}(a_1)]$ is a tournament, $N_A^{++}(a_1)$ dominates $A\setminus A'$, every vertex of $N_A^{++}(a_1)$ has out-degree $7$, and $e(B,A)=10$. Among the at least two out-neighbors of $a_2$ in $A$, choose $a$ sending to another such vertex. This vertex $a$ dominates the three vertices of $A\setminus A'$ and has one more out-neighbor in $A$. Since $a_2$ sends to $a$, it has four or five out-neighbors in $A$, and hence two or three in $B$.

Now $e(B,C)\geq42-15-10=17$. Thus the out-neighbors of $a$ in $B$ collectively dominate $C$. Through $a_1\in A\setminus A'$, the vertex $a$ reaches $a_2$ and at least three vertices of $B$ which it does not dominate. These, together with $C$, give at least seven second out-neighbors of $a$, a contradiction.

\medskip\noindent (ii) Suppose $d^+(a_2)=7$ and $d_{A}^+(a_2)=1$. Then $d_{B}^+(a_2)=6$; write $N^+(a_2)\cap A=\{a_3\}$. There is no arc from $B$ to $a_2$. By Lemma~\ref{lem: boundary of six is at least 6}, every vertex of $N_A^{++}(a_1)$ receives an arc from $B$. Thus $a_2$ has the two vertices of $N_A^{++}(a_1)\setminus\{a_3\}$ and all three vertices of $C$ as second out-neighbors. Consequently, $a_3$ sends to at most one vertex of $A\setminus A'$. Hence,
\[
 e(N_A^{++}(a_1),A)\leq3+2+(1+3+3)=12,
\]
contradicting the lower bound $7+d_{B}^+(a_2)=13$ in Eq.~\eqref{eq: one three B count}.

\medskip\noindent (iii) Suppose $d^+(a_2)=7$ and $d_{A}^+(a_2)=2$. Write $N^+(a_2)\cap A=\{a_3,a_4\}$, $N_A^{++}(a_1)=\{a_3,a_4,a_5\}$, and $B'=N^+(a_2)\cap B=B\setminus\{b\}$. If $b$ is not a second out-neighbor of $a_2$, then no arc from $B'\cup\{a_3,a_4\}$ ends in $b$. All arcs from $B'$ to $A$ end in $N_A^{++}(a_1)$, so
\[
 e(B',N_A^{++}(a_1))\geq35-10-15=10,
 \qquad e(N_A^{++}(a_1),B')\leq15-10=5.
\]
However, Eq.~\eqref{eq: one three A count} gives
$e(N_A^{++}(a_1),B)\geq21-13=8$, and only $a_5$ can send from $N_A^{++}(a_1)$ to $b$. Thus $e(N_A^{++}(a_1),B')\geq7$, a contradiction.

Therefore, $b$ is a second out-neighbor of $a_2$. Since $B'$ collectively dominates $C$, the vertex $a_2$ can have at most two second out-neighbors in $A$. It follows that
\[
 e(\{a_3,a_4\},A)\leq1+2\cdot2=5,
 \qquad e(N_A^{++}(a_1),A)\leq5+6=11.
\]
This contradicts the lower bound $7+d_{B}^+(a_2)=12$ in Eq.~\eqref{eq: one three B count}.

\medskip\noindent (iv) Suppose $d^+(a_2)=7$ and $d_{A}^+(a_2)=3$. Then $N^+(a_2)\cap A=N_A^{++}(a_1)$, $d_{B}^+(a_2)=4$, and $e(N_A^{++}(a_1),A)\geq11$. No vertex of $N_A^{++}(a_1)$ sends to $a_2$. Thus at least eight arcs go from $N_A^{++}(a_1)$ to $A\setminus A'$, and all three vertices of $A\setminus A'$ are second out-neighbors of $a_2$.

Let $B'=N^+(a_2)\cap B$. Its four vertices collectively dominate $C$. The six second out-neighbors already found force the two vertices of $B\setminus B'$ not to be second out-neighbors of $a_2$. Thus no arc from $N_A^{++}(a_1)\cup B'$ ends in $B\setminus B'$. By \textbf{(R1)},
\[
 e(B',N_A^{++}(a_1))\geq28-6-12=10.
\]
On the other hand, Eq.~\eqref{eq: one three A count} gives $e(N_A^{++}(a_1),B)\geq21-12=9$. All these arcs end in $B'$, so \textbf{(R2)} gives $e(B',N_A^{++}(a_1))\leq12-9=3$, a contradiction.
\end{proof}

\begin{lemma}\label{lem: B=6 A1=1}
In $D$, it is not possible to have $|B|=6$ and $|A_1|=1$.
\end{lemma}
\begin{proof}
Lemma~\ref{lem: B=6 A1=1 4 cases} leaves only the pairs $(2,4)$ and $(3,3)$ for $(|N_A^{++}(a_1)|,|C|)$. These are excluded by Lemmas~\ref{lem: one four impossible} and~\ref{lem: one three impossible}, respectively.
\end{proof}

\subsection{Case $|B|=6$, $|A_1|=2$}\label{subsec: B=6, A1=2}
In this subsection, let us assume that $|B|=6$ and $|A_1|=2$. Since $a_1$ has out-degree at least $7$, it has either five or six out-neighbors in $B$. We first show that it cannot have exactly five. After that, we will choose the root $s$ more carefully and use a global counting argument to obtain the final contradiction.

\subsubsection{Establishing $|B_1|=5$ is not possible}
Suppose $|B_1|=5$, and let $B\setminus B_1=\{b\}$. Thus $d^+(a_1)=7$ and $|B_2|\in\{0,1\}$. Throughout this subsubsection, let
\[
 A'=A_1\sqcup N_A^{++}(a_1),\qquad C'=N^+(B_1)\cap C.
\]
Notice that $a_1\in A\setminus A'$, and that $C'$ is exactly the set of second out-neighbors of $a_1$ outside $A\cup B$. Therefore,
\begin{equation}\label{snc2:eq:exact}
 d^{++}(a_1)=|N_A^{++}(a_1)|+|B_2|+|C'|\leq6.
\end{equation}
Also, every vertex of $A$ has at least two out-neighbors in $A$. If it has exactly two, then these two vertices send at least four arcs into $A$, at most one of which stays within their pair. Thus it has at least two second out-neighbors inside $D[A]$. We will use this observation several times.

\begin{lemma}\label{snc2:lem:three}
Let $|B| = 6, |A_1| = 2$ and $|B_1| = 5$. Then we have
\[
 (|N_A^{++}(a_1)|,|B_2|,|C'|)
 \in\{(2,0,4),(2,1,3),(3,1,2)\}.
\]
\end{lemma}
\begin{proof}
We have $2\leq|N_A^{++}(a_1)|\leq4$. Every arc from $A_1$ to $A$ has its head in $A'$, and every arc from $B_1$ to $A$ also has its head in $A'$. If $B_2=\varnothing$, no arc from $A_1\cup B_1$ ends in $b$. Thus,
$$e(A',A)\leq\binom{|N_A^{++}(a_1)|+2}{2}
 +|N_A^{++}(a_1)|(5-|N_A^{++}(a_1)|),$$
 and 
 $$e(A',\{b\})\leq |N_A^{++}(a_1)|+2|B_2|.$$
By \textbf{(R1)},
\[
 e(A',B_1)\geq7|A'|-e(A',A)-e(A',\{b\}).
\]
Next, applying \textbf{(R1)} to $B_1$ and using \textbf{(R2)}, we have
\begin{align*}
 e(B_1,C')+e(B_1,\{b\})
 &\geq35-10-e(B_1,A')\\
 &\geq25-5|A'|+e(A',B_1).
\end{align*}
Substituting the preceding bounds gives
\[
 e(B_1,C')+e(B_1,\{b\})\geq
 \begin{cases}
 19-2|B_2|,&|N_A^{++}(a_1)|=2,\\
 16-2|B_2|,&|N_A^{++}(a_1)|=3,\\
 14-2|B_2|,&|N_A^{++}(a_1)|=4.
 \end{cases}
\]
The left-hand side is at most $5(|C'|+|B_2|)$, which is at most $5(6-|N_A^{++}(a_1)|)$ by Eq.~\eqref{snc2:eq:exact}. If $|N_A^{++}(a_1)|=4$, the lower bound is at least $12$ and the upper bound is $10$, a contradiction. If $|N_A^{++}(a_1)|=3$, then $|B_2|=1$ and $|C'|=2$ are forced. Finally, if $|N_A^{++}(a_1)|=2$, then $|C'|=4$ when $|B_2|=0$, and $|C'|=3$ when $|B_2|=1$.
\end{proof}

\begin{lemma}\label{snc2:lem:y0}
The triple $(|N_A^{++}(a_1)|,|B_2|,|C'|)$ cannot be $(2,0,4)$.
\end{lemma}
\begin{proof}
Suppose otherwise. Name the vertices of $A_1$ as $a_2,a_3$ so that $a_3$ does not send to $a_2$. Then all out-neighbors of $a_3$ in $A$ lie in $N_A^{++}(a_1)$. Since there are exactly two, we have $N^+(a_3)\cap A=N_A^{++}(a_1)$. Also, $a_3$ does not send to $b$, since $B_2=\varnothing$. Its minimum out-degree forces it to dominate $B_1$ and have out-degree $7$.

The vertex $a_3$ has the four vertices of $C'$ and at least two vertices of $A$ as second out-neighbors. Therefore, $b$ is not a second out-neighbor of $a_3$, and no vertex of $N_A^{++}(a_1)$ sends to $b$. Neither does a vertex of $A_1$, so $e(A',\{b\})=0$. Since $|A'|=4$,
\[
 e(A',A)\leq6+2\cdot3=12,\qquad
 e(A',B_1)\geq28-12=16,\qquad
 e(B_1,A')\leq20-16=4.
\]
All out-neighbors of $B_1$ belong to $B_1\cup A'\cup C'$. Hence,
\[
 35\leq e(B_1,B_1)+e(B_1,A')+e(B_1,C')\leq10+4+20=34,
\]
a contradiction.
\end{proof}

For the other two triples, $B_2=\{b\}$. We will need to keep track of the differences in the preceding counts. Let
\begin{align*}
 \theta_{A'}&=\sum_{v\in A'}(d^+(v)-7),&
 \theta_{A',b}&=|A'|-e(A',\{b\}),\\
 \theta_{B_1,C}&=5|C'|-e(B_1,C'),&
 \theta_{B_1,b}&=5-e(B_1,\{b\}).
\end{align*}

 Here $\theta_{A'}$ counts the excess in the out-degrees of $A'$, $\theta_{A',b}$ counts the absent arcs from $A'$ to $b$, and $\theta_{B_1,C},\theta_{B_1,b}$ count the absent arcs from $B_1$ to $C'$ and to $b$, respectively. Moreover,
\begin{equation}\label{snc2:eq:beta-decomp}
 \theta_{A', \bar{A'}}=\binom{|A'|}{2}-e(A',A')+6-e(N_A^{++}(a_1),A\setminus A')=
 \begin{cases}
 12-e(A',A),&|N_A^{++}(a_1)|=2,\\
 16-e(A',A),&|N_A^{++}(a_1)|=3.
 \end{cases}
\end{equation}
Thus $\theta_{A', \bar{A'}}$ counts the missing pairs within $A'$ and the absent arcs from $N_A^{++}(a_1)$ to $A\setminus A'$. Observe that all the above five quantities are nonnegative.

\begin{lemma}\label{snc2:lem:budget}
The above quantities satisfy
\begin{equation}\label{snc2:eq:budget}
 \theta_{A'}+\theta_{A', \bar{A'}}+\theta_{A',b}+\theta_{B_1,C}+\theta_{B_1,b}\leq
 \begin{cases}
 3,&(|N_A^{++}(a_1)|,|C'|)=(2,3),\\
 1,&(|N_A^{++}(a_1)|,|C'|)=(3,2).
 \end{cases}
\end{equation}
If equality holds, then $D[B_1]$ is a tournament, every vertex of $B_1$ has out-degree $7$, and every pair between $A'$ and $B_1$ is joined by an arc. Also,
\begin{equation}\label{snc2:eq:q-out}
 |N^+(b)\setminus(A\cup B)|\geq7-|A\setminus A'|-\theta_{A',b}-\theta_{B_1,b}.
\end{equation}
\end{lemma}
\begin{proof}
All arcs from $A'$ end in $A\cup B_1\cup\{b\}$. Subtracting the arcs to $A$ and to $b$ from its out-degree sum gives
\begin{equation}\label{eq: five A prime B1}
 e(A',B_1)=
 \begin{cases}
 12+\theta_{A'}+\theta_{A', \bar{A'}}+\theta_{A',b},&|N_A^{++}(a_1)|=2,\\
 14+\theta_{A'}+\theta_{A', \bar{A'}}+\theta_{A',b},&|N_A^{++}(a_1)|=3.
 \end{cases}
\end{equation}
On the other hand,
\begin{align*}
 35&\leq\sum_{v\in B_1}d^+(v)
 =e(B_1,B_1)+e(B_1,A')+e(B_1,C'\cup\{b\})\\
 &\leq10+5|A'|-e(A',B_1)+5(|C'|+1)-\theta_{B_1,C}-\theta_{B_1,b}.
\end{align*}
Substitution proves Eq.~\eqref{snc2:eq:budget}. Equality forces the out-degree sum of $B_1$ to be $35$, $e(B_1,B_1)=10$, and $e(B_1,A')+e(A',B_1)=5|A'|$. This proves the additional assertions. Finally, $b$ sends at most $\theta_{A',b}$ arcs to $A'$, at most $\theta_{B_1,b}$ to $B_1$, and at most $|A\setminus A'|$ to $A\setminus A'$. Its out-degree is at least $7$, giving Eq.~\eqref{snc2:eq:q-out}.
\end{proof}

\begin{lemma}\label{snc2:lem:x3}
The triple $(|N_A^{++}(a_1)|,|B_2|,|C'|)$ cannot be $(3,1,2)$.
\end{lemma}
\begin{proof}
Suppose otherwise. Here $|A\setminus A'|=2$. By Eq.~\eqref{snc2:eq:budget}, at most one vertex of $N_A^{++}(a_1)$ can fail to have out-degree $7$, dominate $A\setminus A'$, or send to $b$: every failure contributes at least one to $\theta_{A'}$, $\theta_{A', \bar{A'}}$, or $\theta_{A',b}$, respectively. Choose a vertex $a$ with all three properties.

The vertex $a$ dominates $a_1\in A\setminus A'$. It has three out-neighbors, namely the two vertices of $A\setminus A'$ and $b$, outside the seven-vertex set $N^+(a_1)=A_1\cup B_1$. Thus at most four vertices of $N^+(a_1)$ are out-neighbors of $a$. Since $a\notin N^+(a_1)$, the paths through $a_1$ give at least three second out-neighbors of $a$ in $A_1\cup B_1$.

Also, $\theta_{A',b}+\theta_{B_1,b}\leq1$. Eq.~\eqref{snc2:eq:q-out} shows that $b$ has at least four out-neighbors outside $A\cup B$. These are four more second out-neighbors of $a$, disjoint from the previous three. Hence $a$ is a Seymour vertex, a contradiction.
\end{proof}

We are left with $(|N_A^{++}(a_1)|,|B_2|,|C'|)=(2,1,3)$. Write $A_1=\{a_2,a_3\}$ and without loss of generality assume that $a_3$ does not send to $a_2$. Then
\[
 N^+(a_3)\cap A=N_A^{++}(a_1).
\]
In this case, Eq.~\eqref{snc2:eq:budget} can be strengthened to
\begin{equation}\label{snc2:eq:strict-budget}
 \theta_{A'}+\theta_{A', \bar{A'}}+\theta_{A',b}\leq2.
\end{equation}
Indeed, equality with $3$ would force $\theta_{B_1,C}=\theta_{B_1,b}=0$ and equality in Eq.~\eqref{snc2:eq:budget}. The set $B_1$ would induce a tournament, all its vertices would have out-degree $7$, and it would dominate $C'\cup\{b\}$. Lemma~\ref{snc:lem:common-block} would give a Seymour vertex in $B_1$.

\begin{lemma}\label{snc2:lem:good-vertex}
If $(|N_A^{++}(a_1)|,|B_2|,|C'|)=(2,1,3)$, then no vertex of $N_A^{++}(a_1)$ having out-degree $7$ dominates $A\setminus A'$.
\end{lemma}
\begin{proof}
Suppose such a vertex $a$ exists. Let $d^+_{B}(a)=|N^+(a)\cap B|$, and let $B'=N^+(a)\cap B_1$. Since $a_3$ sends to $a$, the vertex $a$ has at most five out-neighbors in $A$, so $d^+_{B}(a)\geq2$.

The three vertices of $A\setminus A'$ are out-neighbors of $a$ outside $N^+(a_1)$. Since $a$ sends to $a_1$, this gives at least three second out-neighbors of $a$ in $A_1\cup B_1$. 
If $a$ sends to $b$, this lower bound increases to four. If $d^+_{B}(a)=2$, then $a$ dominates all of $A\setminus\{a_3,a\}$, including the other vertex of $N_A^{++}(a_1)$. This gives one additional second out-neighbor in $A_1\cup B_1$ in either case.

First suppose that $a$ does not send to $b$. Then $\theta_{A',b}\geq1$ and $\theta_{B_1,C}+\theta_{B_1,b}\leq2$. If $d^+_{B}(a)\geq3$, the set $B'$ has at least three vertices and collectively dominates all four vertices of $C'\cup\{b\}$. These give four second out-neighbors in addition to the three already found. If $d^+_{B}(a)=2$, its two vertices collectively dominate at least three vertices of $C'\cup\{b\}$, in addition to the four already found. Both possibilities make $a$ a Seymour vertex.

Now suppose that $a$ sends to $b$. If $\theta_{A',b}+\theta_{B_1,b}\leq1$, then Eq.~\eqref{snc2:eq:q-out} gives at least three out-neighbors of $b$ outside $A\cup B$. These, together with the at least four second out-neighbors of $a$ in $A_1\cup B_1$, give a contradiction. Otherwise $\theta_{A',b}+\theta_{B_1,b}\geq2$, and Eq.~\eqref{snc2:eq:budget} gives $\theta_{B_1,C}\leq1$. If $d^+_{B}(a)\geq3$, the set $B'$ has at least two vertices and collectively dominates all three vertices of $C'$. If $d^+_{B}(a)=2$, its one vertex dominates at least two vertices of $C'$. Together with the respective lower bounds four and five in $A_1\cup B_1$, these again make $a$ a Seymour vertex.
\end{proof}

\begin{lemma}\label{lem: five B impossible}
In $D$, if $|B|=6$ and $|A_1|=2$, then $|B_1|=6$.
\end{lemma}
\begin{proof}
Suppose $|B_1|=5$. Lemmas~\ref{snc2:lem:three}, \ref{snc2:lem:y0}, and~\ref{snc2:lem:x3} leave only the triple $(2,1,3)$, with notation as above. By Lemma~\ref{snc2:lem:good-vertex}, each vertex of $N_A^{++}(a_1)$ either has out-degree greater than $7$ or fails to dominate $A\setminus A'$. These two types of contribution are counted by its degree excess and by $\theta_{A', \bar{A'}}$, respectively. Hence,
\[
 \sum_{v\in N_A^{++}(a_1)}(d^+(v)-7)+\theta_{A', \bar{A'}}\geq2.
\]
Together with Eq.~\eqref{snc2:eq:strict-budget}, this forces
\begin{equation}\label{snc2:eq:last-equality}
 \theta_{A'}+\theta_{A', \bar{A'}}=2,\qquad d^+(a_2)=d^+(a_3)=7,
 \qquad\theta_{A',b}=0,\qquad\theta_{B_1,C}+\theta_{B_1,b}\leq1.
\end{equation}
Thus every vertex of $A'$ sends to $b$. Since $N^+(a_3)\cap A=N_A^{++}(a_1)$, there is a unique vertex $b'\in B_1$ such that
\[
 N^+(a_3)=N_A^{++}(a_1)\sqcup(B_1\setminus\{b'\})\sqcup\{b\}.
\]
We claim that $b'$ is a second out-neighbor of $a_3$. Otherwise no arc from $N^+(a_3)$ ends in $b'$, and only $a_2$ among the vertices of $A'$ can send to $b'$. Eq.~\eqref{eq: five A prime B1} gives $e(A',B_1)=14$, so
\[
 e(A',B_1\setminus\{b'\})\geq13.
\]
Removing the four arcs from $a_3$ and applying \textbf{(R2)}, we get
\[
 e(B_1\setminus\{b'\},A'\setminus\{a_3\})\leq12-9=3.
\]
The vertices of $B_1\setminus\{b'\}$ cannot send to $a_3$ or to $b'$. Their remaining out-neighbors lie in their own four-vertex set, in $C'\cup\{b\}$, or in $A'\setminus\{a_3\}$. Thus,
\[
 28\leq\sum_{v\in B_1\setminus\{b'\}}d^+(v)\leq6+16+3=25,
\]
a contradiction.

Finally, $\theta_{A',b}=0$ excludes all arcs from $b$ to $A'$, while $\theta_{B_1,b}\leq1$ allows at most one arc from $b$ to $B_1$. Thus $b$ has at least six out-neighbors outside $B$. None is $a_3$ or an out-neighbor of $a_3$, since the out-neighbors of $a_3$ outside $B$ belong to $A'$, which dominates $b$. These six vertices are second out-neighbors of $a_3$ through $b$. Together with $b'\in B$, they make $a_3$ a Seymour vertex, a contradiction.
\end{proof}

\subsubsection{Structural implications of $|B_1|=6$}
In the previous subsection, we established that, when $|B|=6$ and $|A_1| = 2$, we cannot have $|B_1|=5$. Since the only two values possible for $|B_1|$ are $5$ and $6$, we are forced to have $|B_1|=6$. We will establish some structural properties of $D$ when $|A_1| = 2$ and $|B_1|=6$.

Till now our arguments given are true for any arbitrary
root $s$. Thus, if we make an argument for a vertex, say $s'$, which is possibly other than $s$, but has out-degree exactly $7$, then the previously established statements with respect to $s$ 
will also be valid for $s'$. We will use this trick to change the root sometimes, in this subsection.

\begin{lemma}\label{snc2:lem:abstract-exchange}
Let $X$ be an oriented graph with minimum out-degree at least $2$, and let $X_{d=2}$ be its nonempty set of vertices having out-degree $2$. Suppose that $|N_X^{++}(v)|=2$ for every $v\in X_{d=2}$. Then there is no nonempty set $Y\subseteq V(X)\setminus X_{d=2}$ such that $Y\setminus N_X^+(v)\subseteq N_X^{++}(v)$ for every $v\in X_{d=2}$.
\end{lemma}
\begin{proof}
Suppose such a set $Y$ exists. Fix $v\in X_{d=2}$, and write $N_X^+(v)=\{u,w\}$. If $u\in Y$, then $d_X^+(u)\geq3$. All out-neighbors of $u$ lie in the three-vertex set $\{w\}\cup N_X^{++}(v)$, so $u$ dominates this set. Consequently, $w$ cannot send to $v$ or $u$, and all its out-neighbors belong to $N_X^{++}(v)$. Thus $w\in X_{d=2}$ and $N_X^+(w)=N_X^{++}(v)$. However, $u$ is neither an out-neighbor nor a second out-neighbor of $w$, because it dominates $w$ and both vertices of $N_X^+(w)$. This contradicts the assumed property of $Y$. Therefore, no arc goes from $X_{d=2}$ to $Y$.

Now choose $v\in X_{d=2}$ and $y\in Y$. We must have $y\in N_X^{++}(v)$. Let $u\in N_X^+(v)$ send to $y$, and let $w$ be the other out-neighbor of $v$. Since there is no arc from $X_{d=2}$ to $Y$, we have $u\notin X_{d=2}$ and $d_X^+(u)\geq3$. The same argument shows that $u$ dominates $\{w\}\cup N_X^{++}(v)$, while $w\in X_{d=2}$ and $N_X^+(w)=N_X^{++}(v)$. In particular, $w$ sends to $y$, contradicting the absence of arcs from $X_{d=2}$ to $Y$.
\end{proof}

For a chosen root $s$, retain the sets $A,B,C$ from the framework, and let
\[
 A_{d=2}=\{a\in A:|N^+(a)\cap A|=2\}.
\]
The next lemma summarizes the restrictions valid at every such root.

\begin{lemma}\label{snc:ext:root-facts}
In $D$, for every vertex $s$ having out-degree $7$, the following holds:
\begin{enumerate}
\item[(i)] $|A|=7$, $|B|=6$, and the minimum out-degree of $D[A]$ is $2$;
\item[(ii)] every vertex of $A_{d=2}$ dominates $B$ and has out-degree $8$ in $D$;
\item[(iii)] every vertex of $A$ having out-degree $7$ in $D$ has at least three out-neighbors in $A$;
\item[(iv)] $|C|\in\{3,4\}$.
\end{enumerate}
\end{lemma}
\begin{proof}
By Lemmas~\ref{lem: new tools}(i) and~\ref{lem: B=5}, we have $|B|=6$. Lemmas~\ref{lem: B=6 A1=3} and~\ref{lem: B=6 A1=1} exclude $|A_1|=3$ and $|A_1|=1$, respectively. Thus $|A_1|=2$, proving (i).

By Lemma~\ref{lem: five B impossible}, $a_1$ dominates $B$. Recall that among the vertices having two out-neighbors in $A$, the vertex $a_1$ was chosen with the fewest out-neighbors in $B$. Hence every vertex of $A_{d=2}$ dominates $B$, proving (ii). Part (iii) follows immediately.

For (iv), Lemma~\ref{lem: counting tools}(iv) gives $|C|\geq3$. Every vertex of $A_{d=2}$ has at least two second out-neighbors inside $D[A]$, and all vertices of $C$ as second out-neighbors through $B$. Its out-degree is $8$, so $|C|\leq5$.

Suppose $|C|=5$, and put $Y=N^+(B)\cap A$. The set $Y$ is disjoint from $A_{d=2}$, because $A_{d=2}$ dominates $B$. It is nonempty, since otherwise $\partial^+B=C$ would have size $5$, contrary to Lemma~\ref{lem: boundary of six is at least 6}. For every $a\in A_{d=2}$, we must have
\[
 N_{D[A]}^{++}(a)=N_A^{++}(a),\qquad |N_{D[A]}^{++}(a)|=2.
\]
Also, every vertex of $Y\setminus N^+(a)$ belongs to this second out-neighborhood, being reached through $B$. Applying Lemma~\ref{snc2:lem:abstract-exchange} to $D[A]$ gives a contradiction. Thus $|C|\in\{3,4\}$.
\end{proof}

In particular, if $w$ has out-degree $7$, then $D[N^+(w)]$ has minimum out-degree $2$. Moreover, a vertex of $N^+(w)$ having out-degree $7$ in $D$ has at least three out-neighbors in $N^+(w)$. These conclusions follow by taking $w$ as the root in Lemma~\ref{snc:ext:root-facts}.

\begin{lemma}\label{snc:ext:low-vertices}
Let $a\in A$ have out-degree $7$ in $D$. Then:
\begin{enumerate}
\item[(i)] if $|N^+(a)\cap A|=3$, then $|N_{D[A]}^{++}(a)|\geq2$; if $|N^+(a)\cap A|=5$, then $|N_{D[A]}^{++}(a)|\geq1$;
\item[(ii)] if $|N^+(a)\cap A|=4$ and $|A_{d=2}|\leq2$, then $|N_{D[A]}^{++}(a)|\geq2$;
\item[(iii)] if $|N^+(a)\cap A|=6$, then $|A_{d=2}|\geq3$. Its unique out-neighbor $b$ in $B$ has out-degree $8$ and exactly one out-neighbor in $C$;
\item[(iv)] if $|C|=4$, then $|N^+(a)\cap A|\ne6$.
\end{enumerate}
\end{lemma}
\begin{proof}
\noindent (i) Suppose $|N^+_A(a)|=3$. These three vertices send at least six arcs into $A$, of which at most three stay in $N^+_A(a)$. Hence $a$ has at least one second out-neighbor inside $D[A]$. If it has exactly one, say $a'$, then all these arcs end in $N^+_A(a)\cup\{a'\}$. The upper bound $3+3=6$ forces $D[N^+_A(a)]$ to be a directed triangle dominating $a'$, and every vertex of $N^+_A(a)$ to belong to $A_{d=2}$. Thus $N^+_A(a)$ dominates $B$.

We have $|N^+_B(a)|=4$. There is no arc from $N^+_B(a)$ to $N^+_A(a)$. However, Lemma~\ref{snc:ext:root-facts}, applied with root $a$, requires every vertex of $N^+_B(a)$ to have at least two out-neighbors in $N^+(a)=N^+_A(a)\cup N^+_B(a)$. This would give $e(N^+_B(a),N^+_B(a))\geq8>6$, a contradiction.

If $|N^+_A(a)|=5$ and $N_{D[A]}^{++}(a)=\varnothing$, then $N^+_A(a)$ is forward-closed in $D[A]$. Minimum out-degree $2$ forces $D[N^+_A(a)]$ to be a regular tournament, with $N^+_A(a)\subseteq A_{d=2}$. Thus $N^+_A(a)$ dominates $B$. Now $N^+_B(a)$ has only two vertices, neither of which can have two out-neighbors in $N^+(a)$, again a contradiction.

\medskip\noindent (ii) Suppose $|N^+_A(a)|=4$. Since $2\cdot4>\binom42$, the vertex $a$ has at least one second out-neighbor inside $D[A]$. Suppose it has exactly one, say $a'$. Then $e(N^+_A(a),A)\leq6+4=10$. Since $|A_{d=2}|\leq2$, the sum of the out-degrees in $D[A]$ of the vertices of $N^+_A(a)$ is at least $2+2+3+3=10$. Equality forces exactly two vertices of $A_{d=2}$ in $N^+_A(a)$, a tournament on $N^+_A(a)$, and all four arcs from $N^+_A(a)$ to $a'$. The vertex $a'$ can then have at most two out-neighbors in $A$, so $a'$ is a third vertex of $A_{d=2}$, a contradiction.

\medskip\noindent (iii) Suppose $|N^+_A(a)|=6$. The vertex $a$ dominates $A\setminus\{a\}$, so no vertex of this six-vertex set sends to $a$. If fewer than three of its vertices belonged to $A_{d=2}$, it would send at least $2\cdot2+4\cdot3=16$ arcs to itself, although it can contain at most $15$. Thus $|A_{d=2}|\geq3$.

Let $b$ be the unique out-neighbor of $a$ in $B$. Every vertex of $A_{d=2}$ has two out-neighbors in $A\setminus\{a\}$ and also sends to $b$. Every other vertex of $A\setminus\{a\}$ already has at least three out-neighbors in that set. Thus all six have at least three out-neighbors in $N^+(a)=(A\setminus\{a\})\cup\{b\}$. By Lemma~\ref{snc:ext:root-facts} with root $a$, the remaining vertex $b$ must have exactly two out-neighbors in $N^+(a)$, and it must have out-degree $8$ in $D$.

Through any vertex of $A_{d=2}$, the vertex $a$ reaches all five vertices of $B\setminus\{b\}$. It has no second out-neighbor in $A$, and it can reach $C$ in two steps only through $b$. Hence,
\[
 N^{++}(a)=(B\setminus\{b\})\sqcup(N^+(b)\cap C).
\]
The same root lemma gives $|N^{++}(a)|=6$, so $|N^+(b)\cap C|=1$.

\medskip\noindent (iv) Suppose $|C|=4$ and $|N^+_A(a)|=6$. Since $\partial^+B=N^+_A(B)\sqcup C$, Lemma~\ref{lem: boundary of six is at least 6} gives $|N^+_A(B)|\geq2$. Also $N^+_A(B)\cap A_{d=2}=\varnothing$.

First suppose $a\in N^+_A(B)$. Since $a$ dominates $A\setminus\{a\}$, it is neither an out-neighbor nor a second out-neighbor inside $D[A]$ of a vertex $a'\in A_{d=2}$. However, it is a second out-neighbor of $a'$ through $B$. Since $d^+(a')=8$ and $a'$ also has all four vertices of $C$ as second out-neighbors, we obtain
\[
 |N_{D[A]}^{++}(a')|=2,\qquad
 N_A^{++}(a')=N_{D[A]}^{++}(a')\sqcup\{a\}\qquad(a'\in A_{d=2}).
\]
For every $a'\in A_{d=2}$, each vertex of $N^+_A(B)\setminus\{a\}$ which is not an out-neighbor of $a'$ therefore belongs to $N_{D[A]}^{++}(a')$. The nonempty set $N^+_A(B)\setminus\{a\}\subseteq A\setminus A_{d=2}$ contradicts Lemma~\ref{snc2:lem:abstract-exchange}. Thus $a\notin N^+_A(B)$.

Since $|A_{d=2}|\geq3$ and no arc from $A\setminus\{a\}$ ends in $a$,
\[
 e(A\setminus\{a\},B)\geq42+|A_{d=2}|-15=27+|A_{d=2}|.
\]
Also, no arc from $B$ ends in $a$, so $e(B,A)\leq36-(27+|A_{d=2}|)\leq6$. By (iii), one vertex of $B$ has out-degree $8$ and only one out-neighbor in $C$. Therefore,
\[
 43\leq\sum_{v\in B}d^+(v)
 =e(B,B)+e(B,A)+e(B,C)\leq15+6+(1+5\cdot4)=42,
\]
a contradiction.
\end{proof}

The following count will be useful when we choose the root later.

\begin{lemma}\label{snc:ext:first-neighbourhood-count}
Let $a\in A$ have out-degree $7$ and let $d_8^{B'}$ count the vertices of $N^+_B(a)$ having out-degree at least $8$. Then
\begin{equation}\label{snc:ext:first-neighbourhood-inequality}
 d^+_A(a)(d^{++}_C(a)-2)+\sum_{v\in N^+_A(a)}(d^+(v)-7)\leq d_8^{B'}\leq7-d^+_A(a).
\end{equation}
\end{lemma}
\begin{proof}
We have $N^+(a)=N^+_A(a)\sqcup N^+_B(a)$ and $|N^+_B(a)|=7-d^+_A(a)$. Applying Lemma~\ref{snc:ext:root-facts} with root $a$, each vertex of $N^+_B(a)$ has at least two out-neighbors in $N^+(a)$, and each of its $|N^+_B(a)|-d_8^{B'}$ vertices of out-degree $7$ has at least three. Thus,
\[
 e(N^+_B(a),N^+(a))\geq3(7-d^+_A(a))-d_8^{B'},\qquad
 e(N^+_A(a),N^+(a))\leq21-e(N^+_B(a),N^+(a))\leq3d^+_A(a)+d_8^{B'}.
\]
Every vertex of $N^+(N^+_A(a))\setminus N^+(a)$ is a second out-neighbor of $a$. None is $a$, since $D$ is oriented, and all belong to $A\cup B$, since $N^+_A(a)\subseteq A$. Thus this set is disjoint from $C$ and has at most $6-d^{++}_C(a)$ vertices. Counting the arcs from $N^+_A(a)$, we have
\[
 7d^+_A(a)+\sum_{v\in N^+_A(a)}(d^+(v)-7)
 \leq3d^+_A(a)+d_8^{B'}+d^+_A(a)(6-d^{++}_C(a)).
\]
Rearranging gives the first inequality; the second follows from $d_8^{B'}\leq|N^+_B(a)|=7-d^+_A(a)$. Notice also that $N^{++}(a)\cap C=N^+(N^+_B(a))\cap C$.
\end{proof}

\subsubsection{Showing vertices of out-degree $7$ are independent}\label{sec: out-deg 7 are indep}
Among all vertices of out-degree $7$ in $D$, now choose $s$ so that the number of its out-neighbors having out-degree $7$ is maximum. This is possible since $D$ is finite. Retain the framework at this root, and put
\[
 A_{d=7}=\{a\in A:d^+(a)=7\}.
\]
We will show that $A_{d=7}=\varnothing$. To keep track of the counts, let
\begin{align*}
 \theta_A&=\sum_{a\in A}(d^+(a)-7),&\theta_{A,A}&=21-e(A,A),\\
 \theta_B&=\sum_{b\in B}(d^+(b)-7),&\theta_{B,C}&=6|C|-e(B,C).
\end{align*}
These are nonnegative integers. In particular, $\theta_A\geq7-|A_{d=7}|$ and $\theta_A\geq|A_{d=2}|$, because every vertex of $A_{d=2}$ has out-degree $8$. These definitions are used throughout the rest of the proof.

\begin{lemma}\label{snc:ext:global-count}
For the above root, each $a\in A_{d=7}$ satisfies
\begin{equation}\label{snc:ext:transfer}
 |\{b\in N^+(a)\cap B:d^+(b)\geq8\}|
 \geq |(A\setminus A_{d=7})\setminus N^+(a)|.
\end{equation}
Moreover,
\begin{equation}\label{snc:ext:global-budget}
 \theta_A+\theta_{A,A}+\theta_B+\theta_{B,C}\leq6|C|-13.
\end{equation}
If $A_{d=7}\ne\varnothing$ and no arc goes from $A_{d=7}$ to $A_{d=2}$, then
\begin{equation}\label{snc:ext:charge}
 \theta_A+\theta_{A,A}+\theta_B\geq6.
\end{equation}
\end{lemma}
\begin{proof}
Every $a\in A_{d=7}$ is another possible root. By the choice of $s$, the vertex $a$ has at most $|A_{d=7}|$ out-neighbors of out-degree $7$. Thus at least $7-|A_{d=7}|=|A\setminus A_{d=7}|$ of its out-neighbors have out-degree at least $8$. Those in $A$ belong to $N^+(a)\cap(A\setminus A_{d=7})$, while all its other out-neighbors belong to $B$. This proves Eq.~\eqref{snc:ext:transfer}.

Next, count all arcs having their tail in $A\cup B$. We have
\begin{align*}
 91+\theta_A+\theta_B
 &=e(A,A)+e(A,B)+e(B,A)+e(B,B)+e(B,C)\\
 &\leq(21-\theta_{A,A})+42+15+(6|C|-\theta_{B,C}).
\end{align*}
Rearranging gives Eq.~\eqref{snc:ext:global-budget}.

Finally, suppose $A_{d=7}\ne\varnothing$ and no arc goes from $A_{d=7}$ to $A_{d=2}$. Since $A_{d=2}\subseteq A\setminus A_{d=7}$, Eq.~\eqref{snc:ext:transfer}, applied to any $a\in A_{d=7}$, gives at least $|A_{d=2}|$ out-neighbors of $a$ in $B$ having out-degree at least $8$. Hence $\theta_B\geq|A_{d=2}|$. At most $2|A_{d=2}|$ arcs go from $A_{d=2}$ to $A_{d=7}$, so at least $|A_{d=2}||A_{d=7}|-2|A_{d=2}|$ pairs between these two sets are missing. Thus $\theta_{A,A}\geq|A_{d=2}|(|A_{d=7}|-2)$, and
\begin{align*}
 \theta_A+\theta_{A,A}+\theta_B
 &\geq7-|A_{d=7}|+|A_{d=2}|(|A_{d=7}|-2)+|A_{d=2}|\\
 &=6+(|A_{d=2}|-1)(|A_{d=7}|-1)\geq6.
\end{align*}
\end{proof}

\begin{lemma}\label{snc:ext:positive-tau}
If $A_{d=7}\ne\varnothing$, then $\theta_B\geq1$.
\end{lemma}
\begin{proof}
Suppose $\theta_B=0$. Eq.~\eqref{snc:ext:transfer} then forces every vertex of $A_{d=7}$ to dominate $A\setminus A_{d=7}$. Thus $A\setminus A_{d=7}$ is forward-closed in $D[A]$ and has minimum out-degree at least $2$ in its induced graph. It is nonempty, because it contains $A_{d=2}$. A set with minimum internal out-degree $2$ has at least five vertices, so $|A\setminus A_{d=7}|$ is $5$ or $6$.

If it has five vertices, it induces a regular tournament and is contained in $A_{d=2}$. Choose $a\in A_{d=7}$ that does not send to the other vertex of $A_{d=7}$. Then $N^+(a)\cap A=A\setminus A_{d=7}$, and $N_{D[A]}^{++}(a)=\varnothing$, contradicting Lemma~\ref{snc:ext:low-vertices}(i). If it has six vertices, the unique vertex $a\in A_{d=7}$ has six out-neighbors in $A$. Lemma~\ref{snc:ext:low-vertices}(iii) gives a vertex of $B$ having out-degree $8$, again contradicting $\theta_B=0$.
\end{proof}

\begin{lemma}\label{snc:ext:c3}
If $|C|=3$, then $A_{d=7}= \emptyset$.
\end{lemma}
\begin{proof}
Suppose otherwise. Eq.~\eqref{snc:ext:global-budget} gives $\theta_A+\theta_{A,A}+\theta_B+\theta_{B,C}\leq5$. Since $\theta_A\geq|A_{d=2}|\geq1$ and $\theta_B\geq1$, we obtain $\theta_{B,C}\leq3$.

First, we show that no arc goes from $A_{d=7}$ to $A_{d=2}$. Suppose $a\in A_{d=7}$ sends to some vertex of $A_{d=2}$. 
By Lemma~\ref{snc:ext:root-facts}(iii), $3\leq d^+_A(a)\leq6$. Through its out-neighbor in $A_{d=2}$, the vertex $a$ reaches all $d^+_A(a)-1$ vertices of $B$ which it does not dominate. Its $7-d^+_A(a)$ out-neighbors in $B$ collectively dominate at least
\[
 3-\left\lfloor\frac{\theta_{B,C}}{7-d^+_A(a)}\right\rfloor
\]
vertices of $C$, since each vertex of $C$ they fail to dominate accounts for $7-d^+_A(a)$ absent arcs. These second out-neighbors in $B$ and $C$ are disjoint from $N_{D[A]}^{++}(a)$.

If $d^+_A(a)=3$, Lemma~\ref{snc:ext:low-vertices}(i) gives at least $2+2+3=7$ second out-neighbors of $a$, a contradiction. If $d^+_A(a)=4$ and $\theta_{B,C}\leq2$, the same contradiction follows from $1+3+3=7$; here $|N_{D[A]}^{++}(a)|\geq1$ because its four out-neighbors in $A$ send at least eight arcs into $A$, at most six of which stay among themselves. If $d^+_A(a)=4$ and $\theta_{B,C}=3$, then
\[
 |A_{d=2}|\leq \theta_A\leq5-\theta_B-\theta_{B,C}\leq1.
\]
Lemma~\ref{snc:ext:low-vertices}(ii) now gives at least $2+3+2=7$ second out-neighbors. If $d^+_A(a)=5$, part (i) gives at least $1+4+2=7$.

Finally, if $d^+_A(a)=6$, part (iii) gives $|A_{d=2}|\geq3$ and a vertex of $B$ having out-degree $8$ and only one out-neighbor in $C$. Thus $\theta_A\geq3$, $\theta_B\geq1$, and $\theta_{B,C}\geq2$, contradicting $\theta_A+\theta_{A,A}+\theta_B+\theta_{B,C}\leq5$. This proves that there is no arc from $A_{d=7}$ to $A_{d=2}$.

Eq.~\eqref{snc:ext:charge} therefore gives $\theta_A+\theta_{A,A}+\theta_B\geq6$, contradicting Eq.~\eqref{snc:ext:global-budget} with $|C|=3$.
\end{proof}

That means, if $A_{d=7}\ne\varnothing$, then we are forced to have $|C|=4$. We want to show, when $|C|=4$, then also $A_{d=7}=\emptyset$. For now, to draw a contradiction later, suppose $|C|=4$ and $A_{d=7}\ne\varnothing$.  Since every vertex of $A_{d=2}$ has two out-neighbors in $A$ and dominates $B$, any vertex of $A_{d=2}$ can be chosen as $a_1$ without changing the minimality requirements in the framework. For any such choice, put
\[
 A'=A_1\sqcup N_A^{++}(a_1),\qquad
 \theta_{A'}=\sum_{a\in A'}(d^+(a)-7),
\]
and define
\[
 \theta_{A',A}=
 \begin{cases}
 12-e(A',A),&|N_A^{++}(a_1)|=2,\\
 16-e(A',A),&|N_A^{++}(a_1)|=3.
 \end{cases}
\]
The following lemma justifies these two possibilities. Here $\theta_{A'}$ and $\theta_{A',A}$ refer only to this local choice of $a_1$; $\theta_B$ and $\theta_{B,C}=24-e(B,C)$ retain their definitions at the chosen root.

\begin{lemma}\label{snc:ext:local-count}
For each choice of $a_1\in A_{d=2}$, we have $|N_A^{++}(a_1)|\in\{2,3\}$, $\theta_{A',A}\geq0$, and
\begin{equation}\label{snc:ext:local-budget}
 \theta_{A'}+\theta_{A',A}+\theta_B+\theta_{B,C}\leq
 \begin{cases}
 5,&|N_A^{++}(a_1)|=2,\\
 8,&|N_A^{++}(a_1)|=3.
 \end{cases}
\end{equation}
If $|N_A^{++}(a_1)|=2$ for some choice, then $\theta_{B,C}\leq3$. If $|N_A^{++}(a_1)|=3$, then
\begin{equation}\label{snc:ext:local-beta}
 \theta_{A',A}=(10-e(A',A'))+(6-e(N_A^{++}(a_1),A\setminus A')).
\end{equation}
Moreover, every nonempty set $B'\subseteq B$ satisfies
\begin{equation}\label{snc:ext:reach-C}
 |N^+(B')\cap C|\geq4-\left\lfloor\frac{\theta_{B,C}}{|B'|}\right\rfloor.
\end{equation}
\end{lemma}
\begin{proof}
Every $a_1\in A_{d=2}$ has at least two second out-neighbors inside $D[A]$. Its out-degree is $8$, and all four vertices of $C$ are second out-neighbors through $B$. Thus $2\leq|N_A^{++}(a_1)|\leq3$.

All arcs from $B$ to $A$ end in $A'$, while no arc from $A_1$ ends in $A\setminus A'$. Consequently,
\[
 e(A',A)\leq\binom{|A'|}{2}
 +|N_A^{++}(a_1)|\,|A\setminus A'|.
\]
The right-hand side is $12$ or $16$ according as $|N_A^{++}(a_1)|$ is $2$ or $3$. Thus $\theta_{A',A}\geq0$. Subtracting $e(A',A)$ from the out-degree sum of $A'$ and applying \textbf{(R2)} gives
\[
 e(B,A)=e(B,A')\leq
 \begin{cases}
 8-\theta_{A'}-\theta_{A',A},&|N_A^{++}(a_1)|=2,\\
 11-\theta_{A'}-\theta_{A',A},&|N_A^{++}(a_1)|=3.
 \end{cases}
\]
Now use $42+\theta_B=e(B,A)+e(B,B)+e(B,C)$, together with $e(B,B)\leq15$ and $e(B,C)=24-\theta_{B,C}$, to obtain Eq.~\eqref{snc:ext:local-budget}.

If $|N_A^{++}(a_1)|=2$, choose a vertex of $A_1$ which does not send to the other vertex of $A_1$. Its out-neighbors in $A$ lie in $N_A^{++}(a_1)$, so it belongs to $A_{d=2}$ and has out-degree $8$. Hence $\theta_{A'}\geq1$. Together with $\theta_B\geq1$, this gives $\theta_{B,C}\leq3$.

Eq.~\eqref{snc:ext:local-beta} follows by separating the arcs inside $A'$ from those to $A\setminus A'$. Finally, each vertex of $C$ not collectively dominated by $B'$ requires $|B'|$ absent arcs from $B$ to $C$. This proves Eq.~\eqref{snc:ext:reach-C}.
\end{proof}

\begin{lemma}\label{snc:ext:no-into-K}
If $\theta_{B,C}\leq7$, no vertex of $A_{d=7}$ having exactly three out-neighbors in $A$ sends to $A_{d=2}$. If $\theta_{B,C}\leq5$, there is no arc from $A_{d=7}$ to $A_{d=2}$. In the latter case, for every choice of $a_1\in A_{d=2}$ with $|N_A^{++}(a_1)|=3$,
\begin{equation}\label{snc:ext:coverage}
 \sum_{a\in N_A^{++}(a_1)}(d^+(a)-7)+\theta_{A',A}\geq3.
\end{equation}
\end{lemma}
\begin{proof}
Suppose $a\in A_{d=7}$ sends to $A_{d=2}$. Then $a$ has all its $d^+_A(a)-1$ missing vertices of $B$ as second out-neighbors. If $d^+_A(a)=3$ and $\theta_{B,C}\leq7$, Lemma~\ref{snc:ext:low-vertices}(i) and Eq.~\eqref{snc:ext:reach-C} give at least $2+2+3=7$ second out-neighbors in $A,B,C$, respectively.

If $\theta_{B,C}\leq5$, the values $d^+_A(a)=4$ and $d^+_A(a)=5$ give at least $1+3+3=7$ and $1+4+2=7$ second out-neighbors, respectively. The value $d^+_A(a)=6$ is excluded by Lemma~\ref{snc:ext:low-vertices}(iv), while $d^+_A(a)=2$ would put $a$ in $A_{d=2}$, contrary to $d^+(a)=7$. This proves both assertions about arcs from $A_{d=7}$ to $A_{d=2}$.

Now suppose $|N_A^{++}(a_1)|=3$ and $\theta_{B,C}\leq5$. Each of these three vertices either has out-degree greater than $7$ or does not send to any choice of $a_1 \in A_{d=2}$. A vertex of the former type contributes at least one to the sum in Eq.~\eqref{snc:ext:coverage}. An absent arc of the latter type contributes at least one to $\theta_{A',A}$, because $a_1 \in A\setminus A'$ and Eq.~\eqref{snc:ext:local-beta} counts such absent arcs. Therefore, their total contribution is at least three.
\end{proof}

\begin{lemma}\label{snc:ext:omega-small}
If $|C|=4$, and $A_{d=7}\ne\varnothing$, then we have $\theta_{B,C}\leq3$.
\end{lemma}
\begin{proof}
Suppose $\theta_{B,C}\geq4$. Lemma~\ref{snc:ext:local-count} forces $|N_A^{++}(a_1)|=3$ for every choice of $a_1\in A_{d=2}$. Fix one such choice. Notice that for any $w\in A_{d=7}$, the number $d_8^{B'}$ in Lemma~\ref{snc:ext:first-neighbourhood-count} is at most $\theta_B$, because each of its $d_8^{B'}$ distinct out-neighbors of out-degree at least $8$ in $B$ contributes at least one to $\theta_B$.

\medskip\noindent (i) Suppose $\theta_{B,C}\geq6$. Since $\theta_B\geq1$, Eq.~\eqref{snc:ext:local-budget} gives
\[
 6\leq\theta_{B,C}\leq7,\qquad\theta_{A'}+\theta_{A',A}\leq1,\qquad\theta_B\leq2.
\]
The set $A'$ has five vertices, at most one of which has out-degree greater than $7$. Some vertex $w\in A'$ of out-degree $7$ must have exactly three out-neighbors in $A$. Indeed, otherwise at least four vertices of $A'$ would have at least four out-neighbors in $A$, and the remaining vertex would have at least two. This would give $e(A',A)\geq4\cdot4+2=18>16$.

For this vertex $w$, the set $N^+_B(w)$ has four vertices. By Eq.~\eqref{snc:ext:reach-C}, they collectively dominate at least three vertices of $C$. Lemma~\ref{snc:ext:first-neighbourhood-count}, with $d^+_A(w)=3$ and $d^{++}_C(w) \geq3$, gives $d_8^{B'}\geq3$. This contradicts $d_8^{B'}\leq\theta_B\leq2$.

\medskip\noindent (ii) Suppose $4\leq\theta_{B,C}\leq5$. Eq.~\eqref{snc:ext:coverage}, together with Eq.~\eqref{snc:ext:local-budget} and $\theta_B\geq1$, forces
\[
 \theta_{B,C}=4,\qquad\theta_B=1,\qquad
 \sum_{a\in A_1}(d^+(a)-7)=0.
\]
Thus both vertices of $A_1$ have out-degree $7$. Choose $w\in A_1$ that does not send to the other vertex of $A_1$. All its out-neighbors in $A$ lie in $N_A^{++}(a_1)$, which has size three. By Lemma~\ref{snc:ext:root-facts}(iii), it has exactly three out-neighbors in $A$. Its four out-neighbors in $B$ collectively dominate at least three vertices of $C$, so Lemma~\ref{snc:ext:first-neighbourhood-count} again gives $d_8^{B'}\geq3$. This contradicts $d_8^{B'}\leq\theta_B=1$.
\end{proof}

\begin{lemma}\label{snc:ext:all-high}
If $|C|=4$, then $A_{d=7} = \emptyset$.
\end{lemma}
\begin{proof}
Suppose otherwise, and choose $w\in A_{d=7}$. By Lemma~\ref{snc:ext:omega-small}, $\theta_{B,C}\leq3$. Put $k=|N^+(w)\cap A|$. Lemmas~\ref{snc:ext:root-facts}(iii) and~\ref{snc:ext:low-vertices}(iv) give $d^+_A(w) \in\{3,4,5\}$.

Use the notation $d_8^{B'}$ from Lemma~\ref{snc:ext:first-neighbourhood-count}. If $d^+_A(w)=3$, its four out-neighbors in $B$ collectively dominate all four vertices of $C$, so $d^{++}_C(w)=4$. Eq.~\eqref{snc:ext:first-neighbourhood-inequality} would then give $6\leq d_8^{B'}\leq4$. If $d^+_A(w)=4$, its three out-neighbors in $B$ collectively dominate at least three vertices of $C$, giving $4\leq d_8^{B'}\leq3$. If $d^+_A(w)=5$, its two out-neighbors in $B$ also collectively dominate at least three vertices of $C$, giving $5\leq d_8^{B'}\leq2$. Each possibility is a contradiction.
\end{proof}

\begin{lemma}\label{snc3:cor:independent}
The vertices of out-degree $7$ in $D$ are pairwise non-adjacent.
\end{lemma}

\begin{proof}
For the chosen root, Lemma~\ref{snc:ext:root-facts}(iv) gives $|C|\in\{3,4\}$. Lemmas~\ref{snc:ext:c3} and~\ref{snc:ext:all-high} exclude $A_{d=7}\ne\varnothing$ in both cases. Since the root was chosen to maximize $|A_{d=7}|$, every vertex of out-degree $7$ has no out-neighbor of out-degree $7$. Therefore, no two such vertices are adjacent.
\end{proof}

\subsection{Concluding the proof}
\begin{proof}[Proof of Theorem~\ref{thm:main-existence}]
Recall that $D$ was assumed to be a counterexample. At the chosen root, every vertex of $A$ has out-degree at least $8$, so $\theta_A\geq7$. Eq.~\eqref{snc:ext:global-budget} excludes $|C|=3$, and hence $|C|=4$. Moreover,
\[
 e(A,B)\geq56-21=35,\qquad e(B,A)\leq42-35=7.
\]
Let $t$ be the number of vertices of $B$ having out-degree $7$. The other $6-t$ vertices have out-degree at least $8$. By Lemma~\ref{snc3:cor:independent}, the $\binom t2$ pairs among these $t$ vertices are all nonadjacent. Therefore,
\begin{align*}
 48-t&\leq\sum_{b\in B}d^+(b)
 =e(B,A)+e(B,B)+e(B,C)\\
 &\leq7+15-\binom t2+24=46-\binom t2.
\end{align*}
Rearranging gives
\[
 0\geq2-t+\binom t2
 =1+\frac{(t-1)(t-2)}2\geq1,
\]
where the last inequality holds for every integer $t$. This contradiction shows that $D$ cannot exist, and thus completes the proof.
\end{proof}

\section*{Acknowledgements}
The author Arpan Sadhukhan gratefully acknowledges support from the Anusandhan National Research Foundation (ANRF) through the ANRF-National Post Doctoral Fellowship (N-PDF), File No. PDF/2025/001644. The authors would like to thank Sandeep R. B. for helpful discussions and valuable suggestions. 

\section*{AI Declaration}
We acknowledge the use of ChatGPT 5.5 Plus, ChatGPT 5.5 Pro and ChatGPT 6 Pro for generating key insights in Sections~\ref{subsec: B=6, A1=1} and~\ref{subsec: B=6, A1=2}. For these two sections, our earlier arXiv preprint relied on significant computational assistance and a different proof technique that was devised via substantial help from ChatGPT 5.5 Pro to handle these cases, the current version incorporates a complete handwritten proof for those cases of Seymour's second neighborhood conjecture, produced by ChatGPT 6 Pro with minimal author input and further refined by us. All other sections were developed entirely without AI assistance. This work stems from research that began more than a year ago and was part of the project proposal for the ANRF-NPDF grant of the author Arpan Sadhukhan, received since Jan 2025. We acknowledge the use of AI assistance for verifying the logical steps; its contribution to the writing of the manuscript is negligible. The authors bear full responsibility for the presentation and correctness of all the results.  

\section*{Data availability}
No data were used for the research described in this article.

\bibliographystyle{plain}
\bibliography{refs}

\end{document}